\documentclass[12pt, reqno]{amsart}
\usepackage{amsmath, amsthm, amscd, amsfonts, amssymb, graphicx, color, enumerate, xcolor,  mathrsfs, latexsym}
\usepackage[bookmarksnumbered, colorlinks, plainpages]{hyperref}
\hypersetup{colorlinks=true,linkcolor=red, anchorcolor=green, citecolor=cyan, urlcolor=red, filecolor=magenta, pdftoolbar=true}
\newtheorem{theorem}{Theorem}[section]
\newtheorem{lemma}[theorem]{Lemma}
\newtheorem{proposition}[theorem]{Proposition}
\newtheorem{corollary}[theorem]{Corollary}
\theoremstyle{definition}
\newtheorem{definition}[theorem]{Definition}
\newtheorem{example}[theorem]{Example}
\theoremstyle{remark}
\newtheorem{remark}[theorem]{Remark}

\theoremstyle{question}

\makeatletter
\@namedef{subjclassname@2010}{%
  \textup{2010} Mathematics Subject Classification}
\makeatother
\normalsize
\usepackage{amscd}

\begin{document}
\title[Aron--Berner extensions of triple maps with applications to ....]{Aron--Berner extensions of triple maps with application to the  bidual of    Jordan Banach triple systems}

\author[A.A. Khosravi]{Amin A. Khosravi}
\address{Department of Pure Mathematics, Ferdowsi University of Mashhad, PO Box 1159, Mashhad 91775, Iran.}
\email{amin.khosravi@mail.um.ac.ir}

\author[H.R. Ebrahimi Vishki]{Hamid R. Ebrahimi Vishki}
\address{Department of Pure Mathematics and center of excellence in analysis on algebraic structures (CEAAS), Ferdowsi University of Mashhad, PO Box 1159, Mashhad 91775, Iran.}
\email{vishki@um.ac.ir}

\author[A.M. Peralta]{Antonio M. Peralta}
\address{Departamento de An\`{a}lisis Matem\`{a}tico, Facultad de Ciencias, Universidad de
Granada, 18071 Granada, Spain.}
\email{aperalta@ugr.es}



\begin{abstract}
By extending the notion of Arens regularity of bilinear mappings, we say that a bounded trilinear map on Banach spaces is Aron--Berner regular when all its six Aron--Berner extensions to the bidual spaces coincide. We give some results on the Aron--Berner regularity of certain trilinear maps. We then   focus on the bidual, $E^{**},$  of a Jordan Banach triple system $(E,\pi)$, and investigate those conditions  under which  $E^{**}$ is itself a Jordan Banach triple system under each of  the Aron--Berner extensions of the triple product $\pi.$  We also compare these six triple products with those arising from certain ultrafilters based on the ultrapower formulation of the  principle of local reflexivity.
In particular, we examine the Aron--Berner triple products on the bidual of a  JB$^*$-triple in relation with the so-called Dineen's theorem.  Some illuminating  examples are  included and some questions are also left undecided.
 \end{abstract}
\maketitle

\section{Introduction}

The question whether the bidual space of a Jordan Banach triple system can be made into a Jordan Banach triple system remains as an intriguing problem in the general theory of Jordan Banach triple systems. In the realm of JB$^*$-triples, in 1986,  Dineen \cite{dineen1,dineen2} showed  that the bidual, $E^{**},$ of a JB$^*$-triple $E$ is a JB$^*$-triple itself.  Dineen's method for extending the triple product of the original space to the bidual is based on the so-called principle of local reflexivity developed in ultrafilters theory \cite{He}. The procedure induces an isometric embedding from $E^{**}$ into a complemented subspace of a suitable ultrapower $E_{\mathcal U}$, from which the natural triple product is carried to $E^{**}$ making the latter a JB$^*$-triple. Later, by refining the ultrafilter $\mathcal U$ used in  Dineen's proof,  Barton and  Timoney \cite{BT} showed that the ultrapower triple product also is separately $w^*$-continuous.\smallskip

One of our aims here is to equip the bidual of a general Jordan Banach triple system $(E,\pi)$ with six triple products arising from the so-called Aron--Berner extensions of $\pi$. We then study the relationship between these six extensions with each others and also with those triple products arising from certain ultrapowers of $E$.  For this purpose, we first focus on the six  extensions of a general triple map and study their coincidence.\smallskip

Let us give a short history on the different extensions of multilinear maps to the bidual spaces. The problem of extending multilinear mappings to certain superspaces of a given space was initiated in 1951 by  Arens \cite{Arr, Ar} for bilinear operators. The most influential contribution is due to  Aron and  Berner, who settled in \cite{AB} the most outspread method for extending multilinear operators (see also \cite{ArColGam91}). In 1989, Davie and  Gamelin (see \cite{DG}) studied the isometric nature of the Aron--Berner extensions in the case of the bidual. Different descriptions of the Aron--Berner extensions are gathered and discussed in the survey articles \cite{CGV,Z1}.
\smallskip

The results of this paper are organized  in two sections. Section  \ref{first} is devoted to the study of different methods to produce extensions of a bounded trilinear map to the bidual spaces. Subsection \ref{extension}  introduces these extensions  and contains some preparatory materials. These extensions are inspired  by  Arens' method \cite{Ar}, and they coincide with the extensions constructed by Aron and
Berner for polynomials \cite{AB}.  We say that a trilinear map is Aron--Berner regular when all its six extensions to the bidual spaces coincide. In Section \ref{AB}, we investigate the Aron--Berner regularity of certain trilinear mappings on Banach spaces. First  we prove  our main result of this section (Theorem \ref{c6}) containing some helpful equivalent conditions for  the Aron--Berner regularity of a trilinear  map. We subsequently apply this result to rediscover a theorem due to Bombal and Villanueva, providing sufficient conditions to guarantee the Aron--Berner regularity of some trilinear maps whose certain derived operators are weakly compact (see Theorem \ref{c7}). We culminate this section by presenting the useful Corollary \ref{regular space} that will be crucial for the next section.\smallskip

In Section \ref{second}, we focus our attention on the biduals of (real or complex) Jordan Banach triple systems and JB$^*$-triples. In subsection \ref{second1}, we apply the method outlined in Section \ref{first} to assign six extensions  $\pi^0,$ $\pi^1,$ $\pi^2,$ $\pi^3,$ $\pi^4$, and $\pi^5$ (termed Aron--Berner triple products) to the bidual, $E^{**},$ of a Jordan Banach triple system $(E,\pi)$, and we study their properties in line with  the following two questions:
 \begin{enumerate}[\hspace{2em}$\bullet$]
 \item When is a Jordan Banach triple system $(E,\pi)$ Aron--Berner regular?
 \item Is  $(E^{**},\pi^i)$ ($i=0,\ldots,5$)  a Jordan Banach triple system?
 \end{enumerate}
We examine the first question in Example \ref{not} illustrating that a Jordan Banach triple system need not be generally  Aron--Berner regular. However, the situation is different for a JB$^*$-triple, as we rediscover in Corollary \ref{regularity} that every JB$^*$-triple is Aron--Berner regular. A partial answer is also provided by Proposition \ref{p unitary}. The second question is also examined and answered in negative by the same Example \ref{not}.  The main discrepancy arises form the fact that, for $i\in\{0,1,\ldots,5\}$, the Aron--Berner extension $\pi^i$ on $E^{**}$ is not, in general, symmetric in the outer variables, and this has some unpleasant consequences when dealing with the bidual. In Proposition \ref{symmetry}, we investigate when an extension $\pi^i$ is symmetric in the outer variables, and then in Theorem \ref{t outer symmetric is sufficient and necessary for pi0} we show that $(E^{**},\pi^0)$ is a Jordan Banach triple system whenever $\pi^0$ (equivalently, $\pi^2$) is symmetric in the outer variables; although Example \ref{example pi1 symmetric but not Jordan identity} shows that this is not the case for other extensions.\smallskip 

 In the context of JB$^*$-triple, however, we expect a more  rich  structure for the bidual $E^{**}$. In this regard, we bring  Theorem \ref{e13} as a consequence of Dineen's theorem confirming  that the bidual of a JB$^*$-triple is a JBW$^*$-triple under each Aron--Berner triple product.
Finally in subsection \ref{second2}, we do some comparisons between the Aron--Berner triple products and  the ultrapower triple products and we also highlight some advantages of the Aron--Berner triple products. We conclude the paper with some unanswered questions in the final section.

\section{{\sc Extensions of triple maps to the bidual spaces}}\label{first}

Although the problem of extending bounded bilinear maps was originated by the pioneer works of Arens \cite{Arr, Ar}, the most successful method for multilinear maps is due to Aron and Berner \cite{AB}, who showed that such extensions always exist (see also Davie and Gamelin \cite{DG} and \cite{ArColGam91}). Before going on to the main aim of this section, we succinctly present a description of the six Aron--Berner extensions of a trilinear map in subsection \ref{extension}, and we postpone the main results of this section to subsection \ref{AB}.

\subsection{Aron--Berner extensions of trilinear maps}\label{extension}

Let $X_{1} , X_{2} , X_{3},$ and $ Y $ be Banach   spaces. We define  the  adjoint $f^*:Y^*\times X_1\times X_2\to X_3^*$ of a bounded trilinear  mapping $ f:X_{1} \times X_{2}\times X_{3} \longrightarrow Y$ by
 \begin{align}\label{11}
 \langle f^*(y^*,x_1,x_2),x_3\rangle=\langle y^*,f(x_1,x_2,x_3)\rangle,\qquad (x_1\in X_1,x_2\in X_2, x_3\in X_3, y^*\in Y^*).
\end{align}
The iterated adjoins of $f$ can also be defined by  $f^{**}:=(f^*)^*, f^{***}:=(f^{**})^*, f^{****}:=(f^{***})^*$, and so on. For the sake of simplicity, we use the symbol``$^\circledast$'' for ``$^{****}$''
and we also identify a  Banach space  $X$ with its canonical image in its bidual $X^{\ast\ast}.$
\smallskip
It   is not hard to see that  the bounded trilinear  map $$f^{\circledast}: X_{1}^{**} \times X_{2}^{**}\times X_{3}^{**} \longrightarrow Y^{**}$$ is a norm preserving extension of $f$ and it is the unique extension of $f$ to the bidual spaces with the property that,
for every  fixed   $ x_{2}^{\ast\ast}\in  X_{2}^{\ast\ast}$, $x_{3}^{\ast\ast} \in X_{3}^{\ast\ast}$, $x_{1}\in X_{1}$, and $x_{2}\in X_{2}$, the  maps
\begin{align}\label{identity}
  \cdot\longmapsto f^{\circledast}(\cdot, x_{2}^{\ast\ast}, x_{3}^{\ast\ast}),\quad
  \cdot\longmapsto f^{\circledast}(x_{1},\cdot ,x_{3}^{\ast\ast}),\quad
  \cdot\longmapsto f^{\circledast}(x_{1},x_{2},\cdot)
  \end{align}
 are $w^{\ast}$-continuous on $X_1^{**}, X_2^{**}$, and $X_3^{**}$, respectively.\smallskip

Similarly, by changing the usual ordering in  the involved three variables,  one can obtain six, generally different, extensions of $f$ to the bidual spaces. More precisely, for every permutation $\sigma\in S_3,$ where
\[S_{3}:=\{\sigma_0:=(),\sigma_1:=(23),\sigma_2:=(13),\sigma_3:=(12),\sigma_4:=(132),\sigma_5:=(123)\}\]
is  the symmetric group of  all permutations on the three elements  set $\{1,2,3\},$
 if we  define $ f^{\sigma}: X_{\sigma(1)}\times X_{\sigma(2)} \times X_{\sigma(3)} \longrightarrow Y$ by
 \[f^{\sigma}(x_{\sigma(1)},x_{\sigma(2)},x_{\sigma(3)})= f(x_{1},x_{2},x_{3}),\quad (x_{\sigma(i)}\in X_{\sigma(i)}, i=1,2,3),\]
 then $f^{\sigma}$ is a bounded  trilinear map with the same norm of $f$ and, it is easy to verify that
\[f^{\sigma \circledast\sigma ^{-1}}=((f^\sigma)^\circledast)^{\sigma^{-1}}:  X_{1}^{\ast\ast}\times X_{2}^{\ast\ast}\times X_{3}^{\ast\ast}\longrightarrow Y^{\ast\ast}\]
also is a norm preserving multilinear extension of $f$ to the bidual spaces.\smallskip

Using a standard procedure, based on the fact that every Banach space is $w^*$-dense in its bidual, if $({x_1}_{\alpha_1}), ({x_2}_{\alpha_2}), ({x_3}_{\alpha_3})$ are, respectively, bounded nets in $X_1, X_2, X_3, $  $w^*$-converging  to $x_1^{**},x_2^{**},x_3^{**}$, one can easily verify that
\begin{align}\label{itrated}
f^{\sigma \circledast\sigma ^{-1}}(x_{1}^{**},x_{2}^{**},x_{3}^{**})=w^*\mbox{-}\lim_{\alpha_{\sigma(1)}} \lim_{\alpha_{\sigma(2)}}\lim_{\alpha_{\sigma(3)}} f({x_1}_{\alpha_1},{x_2}_{\alpha_2},{x_3}_{\alpha_3}),
\end{align}
where by ``$w^*\mbox{-}\lim_{\alpha_{\sigma(1)}} \lim_{\alpha_{\sigma(2)}}\lim_{\alpha_{\sigma(3)}}$'' we mean that all of the above limits are taken in the $w^*$ topology.    \smallskip

The crucial property of $f^{\circledast}$ described in \eqref{identity} admits an appropriate reformulation for the extension $f^{\sigma \circledast\sigma ^{-1}}$. To avoid ambiguity, it should be noted that, for a specific $\sigma$, each  component of the triple  $(\cdot_{\sigma(1)},\cdot_{\sigma(2)}, \cdot_{\sigma(3)})$ will jump  in its right place. For example, $(\cdot_{\sigma_5(1)}, x_{\sigma_5(2)}^{\ast\ast},x_{\sigma_5(3)}^{\ast\ast})=(x_{1}^{\ast\ast},\cdot, x_{3}^{\ast\ast}).$ The concrete property is presented in the next result.

\begin{proposition}\label{B2} Let  $ f:X_{1} \times X_{2}\times X_{3} \longrightarrow Y $ be a bounded trilinear map. Then, for every $\sigma\in S_3,$ the map $f^{\sigma\circledast\sigma ^{-1}}$
   is the   unique extension of $f$ to the bidual spaces with the property that, for every  fixed  $ x_{\sigma(2)}^{\ast\ast}\in  X_{\sigma(2)}^{\ast\ast}$, $x_{\sigma(3)}^{\ast\ast} \in X_{\sigma(3)}^{\ast\ast}$, $x_{\sigma(1)}\in X_{\sigma(1)}$, and $x_{\sigma(2)}\in X_{\sigma(2)}$, the   maps
\begin{align*}
\cdot\longmapsto f^{\sigma\circledast\sigma ^{-1}}(\cdot_{\sigma(1)}, x_{\sigma(2)}^{\ast\ast},x_{\sigma(3)}^{\ast\ast}):X_{\sigma(1)}^{**}\longrightarrow Y^{**},\\
\cdot\longmapsto f^{\sigma\circledast\sigma ^{-1}}(x_{\sigma(1)},\cdot_{\sigma(2)} ,x_{\sigma(3)}^{\ast\ast}): X_{\sigma(2)}^{**}\longrightarrow Y^{**},\\
\cdot\longmapsto f^{\sigma\circledast\sigma ^{-1}}(x_{\sigma(1)},x_{\sigma(2)},\cdot_{\sigma(3)}):X_{\sigma(3)}^{**}\longrightarrow Y^{**},
 \end{align*}
 are $w^{\ast}\mbox{-}w^{\ast}$-continuous.\smallskip
\end{proposition}

For instance, the trilinear map  $f^{\sigma_5\circledast\sigma_5^{-1}}$ is the unique extension of $f$ such that, for every fixed $ x_{3}^{\ast\ast}\in  X_{3}^{\ast\ast}$, $x_{1}^{\ast\ast} \in X_{1}^{\ast\ast}$, $x_{2}\in X_{2}$, and $x_{3}\in X_{3}$, the  maps
\[\cdot\longmapsto f^{\sigma_5\circledast\sigma_5^{-1}}(x_{1}^{\ast\ast},\cdot, x_{3}^{\ast\ast}),\quad \cdot\longmapsto f^{\sigma_5\circledast\sigma_5^{-1}}(x_{1}^{\ast\ast}, x_{2},\cdot),\ \ \cdot\longmapsto f^{\sigma_5\circledast\sigma_5^{-1}}(\cdot, x_{2},x_{3})\]
 are $w^{\ast}\mbox{-}w^{\ast}$-continuous on $ X_{2}^{\ast\ast}\longrightarrow Y^{**},  X_{3}^{\ast\ast}\longrightarrow Y^{**},  X_{1}^{\ast\ast}\longrightarrow Y^{**},$ respectively.\medskip

\begin{definition}\label{B5}
A bounded trilinear map $f: X_{1}\times X_{2} \times X_{3}\longrightarrow Y $  is said to be Aron--Berner  regular if its six Aron--Berner extensions to the bidual spaces coincide, that is, $f^{\sigma \circledast \sigma ^{-1}}=f^{\circledast} $ for each $\sigma \in S_{3}$.
\end{definition}

To give a concrete example, take a normed space $X$ and $\phi, \psi\in X^*$. Then the trilinear map
$f: X\times X \times X\longrightarrow X$   defined by
\[(x_{1},x_{2},x_{3})\mapsto f(x_{1},x_{2},x_{3})=\langle\phi ,x_{1}\rangle\langle\psi , x_{2}\rangle x_{3}: X\times X \times X\longrightarrow X\qquad (x_{1},x_{2},x_{3}\in X),\]
 is Aron--Berner regular. Indeed, a direct verification reveals that
\[f^{\circledast}(x_{1}^{**},x_{2}^{**},x_{3}^{**})=\langle x_{1}^{\ast\ast} ,\phi\rangle\langle x_{2}^{\ast\ast} ,\psi\rangle x_{3}^{\ast\ast}\qquad (x_{1}^{**},x_{2}^{**},x_{3}^{**}\in X^{**}).\]\smallskip
{\begin{remark}\label{r new} Suppose $ f:X_{1} \times X_{2}\times X_{3} \longrightarrow Y$ admits a norm preserving extension $F:X_{1}^{**} \times X_{2}^{**}\times X_{3}^{**} \longrightarrow Y^{**}$ (produced by any method, Aron--Berner or any other) which is separately $w^*$-continuous. Then it can be easily deduced form \eqref{itrated} that $$f^{\sigma \circledast\sigma ^{-1}}(x_{1}^{**},x_{2}^{**},x_{3}^{**})= F(x_{1}^{**},x_{2}^{**},x_{3}^{**}),$$ for all $(x_{1}^{**},x_{2}^{**},x_{3}^{**})\in X_{1}^{**} \times X_{2}^{**}\times X_{3}^{**}$ (this was observed, for example, in \cite[page 79]{ArColGam91}, and with other words in \cite[\S 1.1]{PVMY}). Consequently, if there exists an extension of $f$ to a trilinear mapping $F:X_{1}^{**} \times X_{2}^{**}\times X_{3}^{**} \longrightarrow Y^{**}$ which is separately $w^*$-continuous, then all Aron--Berner extensions of $f$ coincide with $F$, so  $f$ is Aron--Berner regular. Reciprocally, if $f$ is Aron--Berner regular then all its Aron--Berner extensions coincide, and thus Proposition \ref{B2} assures that $f^{\circledast}$ is separately $w^*$-continuous. The same equivalence holds for multi-sesquilinear maps too, in the complex setting by multi-sesquilinear maps we mean maps which are conjugate-linear at some variables.
\end{remark}}

\subsection{Aron--Berner regularity of  trilinear maps}\label{AB}
We are now preparing to present the main result  of this section concerning the Aron--Berner regularity of a trilinear map (see Theorem \ref{c6}). For this purpose, we need and provide the following  lemmas which are also interesting by their own right.\smallskip

By applying Proposition \ref{B2}, it is not hard to check that all the six Aron--Berner extensions $f^{\sigma\circledast\sigma^{-1}}:X_{1}^{\ast\ast}\times X_{2}^{\ast\ast}\times X_{3}^{\ast\ast}\longrightarrow Y^{**}, (\sigma\in S_3)$,  coincide  on each one of the following subsets:
 \begin{align}\label{B8}
X_{1}\times X_{2}\times X_{3}^{\ast\ast},\quad  X_{1}\times X_{2}^{\ast\ast}\times X_{3},\quad {\rm and}\ X_{1}^{\ast\ast}\times X_{2}\times X_{3}.
 \end{align}
 We use this fact, in the following lemma, to establish more identities of the various extensions $f^{\sigma\circledast\sigma^{-1}}$ on certain subsets of  $X_{1}^{\ast\ast}\times X_{2}^{\ast\ast}\times X_{3}^{\ast\ast}.$
Before proceeding, to simplify the notation, we note that $f^{\sigma_0\circledast\sigma_0^{-1}}=f^{\circledast},$ and that  $\sigma_i^{-1}=\sigma_i$ for $i=0,1,2,3$, while $\sigma_4^{-1}=\sigma_5.$

\begin{lemma}\label{B9} Let  $ f:X_{1} \times X_{2}\times X_{3} \longrightarrow Y $ be a bounded trilinear mapping and let  $x_i\in X_i, x_i^{**}\in X_i^{**}$ for  $i=1,2,3.$ Then the following identities hold:
\begin{enumerate}[$(a)$]
\item $f^{\circledast}(x_{1},x_{2}^{\ast\ast},x_{3}^{\ast\ast})=f^{\sigma_{3}\circledast\sigma_{3}}(x_{1},x_{2}^{\ast\ast},x_{3}^{\ast\ast})=
f^{\sigma_{5}\circledast\sigma_{4}}(x_{1},x_{2}^{\ast\ast},x_{3}^{\ast\ast})\ and \\
f^{\sigma_{1}\circledast\sigma_{1}}(x_{1},x_{2}^{\ast\ast},x_{3}^{\ast\ast})=
f^{\sigma_{2}\circledast\sigma_{2}}(x_{1},x_{2}^{\ast\ast},x_{3}^{\ast\ast})=
f^{\sigma_{4}\circledast\sigma_{5}}(x_{1},x_{2}^{\ast\ast},x_{3}^{\ast\ast}) ;$\\
\item $f^{\circledast}(x_{1}^{\ast\ast},x_{2},x_{3}^{\ast\ast})=f^{\sigma_{1}\circledast\sigma_{1}}(x_{1}^{\ast\ast},x_{2},x_{3}^{\ast\ast})=
f^{\sigma_{3}\circledast\sigma_{3}}(x_{1}^{\ast\ast},x_{2},x_{3}^{\ast\ast})\ and\\
f^{\sigma_{2}\circledast\sigma_{2}}(x_{1}^{\ast\ast},x_{2},x_{3}^{\ast\ast})=
f^{\sigma_{4}\circledast\sigma_{5}}(x_{1}^{\ast\ast},x_{2},x_{3}^{\ast\ast})=
f^{\sigma_{5}\circledast\sigma_{4}}(x_{1}^{\ast\ast},x_{2},x_{3}^{\ast\ast});$\\
\item $f^{\circledast}(x_{1}^{\ast\ast},x_{2}^{\ast\ast},x_{3})=f^{\sigma_{1}\circledast\sigma_{1}}(x_{1}^{\ast\ast},x_{2}^{\ast\ast},x_{3})=
f^{\sigma_{4}\circledast\sigma_{5}}(x_{1}^{\ast\ast},x_{2}^{\ast\ast},x_{3})\ and \\
f^{\sigma_{2}\circledast\sigma_{2}}(x_{1}^{\ast\ast},x_{2}^{\ast\ast},x_{3})=
f^{\sigma_{3}\circledast\sigma_{3}}(x_{1}^{\ast\ast},x_{2}^{\ast\ast},x_{3})=
f^{\sigma_{5}\circledast\sigma_{4}}(x_{1}^{\ast\ast},x_{2}^{\ast\ast},x_{3}).$\\
\end{enumerate}
\end{lemma}

\begin{proof}
We give a proof for $(a)$, other parts can be obtained similarly . Fix  $x_1\in X_1,x_2^{**}\in X_2^{**},$ and $ x_3^{**}\in X_3^{**}.$ Since all $f^{\sigma\circledast\sigma^{-1}}$ coincide on $X_{1}\times X_{2}\times X_{3}^{\ast\ast}$ and,  by Proposition \ref{B2},  the map $\cdot\mapsto f^{\sigma\circledast\sigma^{-1}}(x_{1},\cdot,x_3^{**})$ is $w^*\text{-}w^*$-continuous for $\sigma\in\{\sigma_0,\sigma_3,\sigma_5\}$,  we get, via Goldstine's theorem, the coincidence of $f^{\sigma\circledast\sigma^{-1}} $ for $\sigma\in\{\sigma_0,\sigma_3,\sigma_5\}$, on $X_{1}\times X_{2}^{\ast\ast}\times X_{3}^{\ast\ast}$. Similarly, since all $f^{\sigma\circledast\sigma^{-1}}$ coincide on $X_{1}\times X_{2}^{\ast\ast}\times X_{3}$ and,  by Proposition  \ref{B2},  the map $\cdot\mapsto f^{\sigma\circledast\sigma^{-1}}(x_{1},x_2^{**},\cdot)$ is $w^*\text{-}w^*$-continuous for $\sigma\in\{\sigma_1,\sigma_2,\sigma_4\}$, we obtain the coincidence of  $f^{\sigma\circledast\sigma^{-1}}$  on $X_{1}\times X_{2}^{\ast\ast}\times X_{3}^{\ast\ast}$ for every $\sigma\in\{\sigma_1,\sigma_2,\sigma_4\}$.
\end{proof}

In the following lemma, we apply  Lemma \ref{B9} to  show that the Aron--Berner regularity of $f$  is equivalent to coincidence of only  certain extensions of $f$. We notice  that there are eight triplets of permutations in $S_3$ with different values at $1$, which are listed as follows:
 \[(\sigma_0,\sigma_2,\sigma_3), (\sigma_0,\sigma_2,\sigma_5), (\sigma_0,\sigma_3,\sigma_4), (\sigma_0,\sigma_4,\sigma_5), (\sigma_1,\sigma_2,\sigma_3), (\sigma_1,\sigma_2,\sigma_5), (\sigma_1,\sigma_3,\sigma_4)\ \mbox{and}\ (\sigma_1,\sigma_4,\sigma_5).\]
\begin{lemma}\label{c1} A bounded trilinear map $ f : X_{1}\times X_{2} \times X_{3}\longrightarrow Y $ is Aron--Berner regular if and only if $f^{\tau \circledast \tau ^{-1}}=f^{\rho \circledast\rho^{-1}}=f^{\eta\circledast\eta^{-1}} $ for some permutations $\tau, \rho, \eta$ in $S_3$ with different values at $1$.
\end{lemma}
 \begin{proof}
The necessity is clear from definition. For the sufficiency,  according to the above comments, there are eight triples $(\tau,\rho,\eta)$ of permutations with different values at $1$. We only give the proof for the triple $(\tau,\rho,\eta)=(\sigma_1,\sigma_4,\sigma_5).$ Suppose that   $f^{\sigma_{1} \circledast \sigma_{1}}=f^{\sigma_{4} \circledast \sigma_{5}}=f^{\sigma_{5} \circledast \sigma_{4}} $.  We are going to show that $f^{\circledast}=f^{\sigma_{1} \circledast \sigma_{1}}$, $f^{\sigma_{3} \circledast \sigma_{3}}=f^{\sigma_{5} \circledast \sigma_{4}}$, and $f^{\sigma_{2} \circledast \sigma_{2}}=f^{\sigma_{4} \circledast \sigma_{5}}$.\medskip

For $i=1,2, 3$, let $x_{i}^{\ast\ast}\in X_{i}^{\ast\ast}$ and let $({x_i}_{\alpha})$ be a net in $X_i$, $w^*$-converging to  $x_{i}^{\ast\ast}$. Then
\begin{align}
f^{\circledast}(x_{1}^{\ast\ast},x_{2}^{\ast\ast},x_{3}^{\ast\ast})&=w^{\ast}\mbox{-}\lim_\alpha f^{\circledast}({x_1}_{\alpha},x_{2}^{\ast\ast},x_{3}^{\ast\ast}) \nonumber \tag{$\sigma_{0}(1)=1$}\\
&= w^{\ast}\mbox{-}\lim_\alpha f^{\sigma_{5} \circledast \sigma_{5}}({x_1}_{\alpha},x_{2}^{\ast\ast},x_{3}^{\ast\ast})\tag{Lemma \ref{B9}(a)}\\
&= w^{\ast}\mbox{-}\lim_\alpha f^{\sigma_{1} \circledast \sigma_{1}}({x_1}_{\alpha},x_{2}^{\ast\ast},x_{3}^{\ast\ast})\nonumber  \\
&= f^{\sigma_{1} \circledast \sigma_{1}}(x_{1}^{\ast\ast},x_{2}^{\ast\ast},x_{3}^{\ast\ast}), \tag{$\sigma_{1}(1)=1$}
\end{align}
\begin{align}
f^{\sigma_{3} \circledast \sigma_{3}}(x_{1}^{\ast\ast},x_{2}^{\ast\ast},x_{3}^{\ast\ast})&= w^{\ast}\mbox{-}\lim_\alpha f^{\sigma_{3} \circledast \sigma_{3}}(x_{1}^{\ast\ast},{x_2}_{\alpha},x_{3}^{\ast\ast})\tag{$\sigma_{3}(1)=2$} \\
&=w^{\ast}\mbox{-}\lim_\alpha f^{\sigma_{1} \circledast \sigma_{1}}(x_{1}^{\ast\ast},{x_2}_{\alpha},x_{3}^{\ast\ast})\tag{Lemma \ref{B9}(b)} \\
&=w^{\ast}\mbox{-}\lim_\alpha f^{\sigma_{5} \circledast \sigma_{5}}(x_{1}^{\ast\ast},{x_2}_{\alpha},x_{3}^{\ast\ast})\nonumber \\
&= f^{\sigma_{5} \circledast \sigma_{4}}(x_{1}^{\ast\ast},x_{2}^{\ast\ast},x_{3}^{\ast\ast}),\tag{$\sigma_{5}(1)=2$}
\end{align}
and finally
\begin{align}
f^{\sigma_{2} \circledast \sigma_{2}}(x_{1}^{\ast\ast},x_{2}^{\ast\ast},x_{3}^{\ast\ast})&= w^{\ast}\mbox{-}\lim_\alpha f^{\sigma_{2} \circledast \sigma_{2}}(x_{1}^{\ast\ast},x_{2}^{\ast\ast},{x_3}_{\alpha}), \tag{$\sigma_{2}(1)=3$} \\
&= w^{\ast}\mbox{-}\lim_\alpha f^{\sigma_{5} \circledast \sigma_{4}}(x_{1}^{\ast\ast},x_{2}^{\ast\ast},{x_3}_{\alpha}), \tag{Lemma \ref{B9}(c)} \\
&= w^{\ast}\mbox{-}\lim_\alpha f^{\sigma_{4} \circledast \sigma_{5}}(x_{1}^{\ast\ast},x_{2}^{\ast\ast},{x_3}_{\alpha})\nonumber \\
&=f^{\sigma_{4} \circledast \sigma_{5}}(x_{1}^{\ast\ast},x_{2}^{\ast\ast},x_{3}^{\ast\ast}).\tag{$\sigma_{4}(1)=3$}
\end{align}
\end{proof}

We are now ready to present the main result of this section characterizing the Aron--Berner regularity of a bounded trilinear map.

\begin{theorem}\label{c6}
For every bounded trilinear map $ f:X_{1} \times X_{2}\times X_{3} \longrightarrow Y$, the following statements are equivalent:
\begin{enumerate}[$(a)$]
  \item $f$ is Aron--Berner regular;
  \item  $f$  admits a norm preserving extension $F:X_{1}^{**} \times X_{2}^{**}\times X_{3}^{**} \longrightarrow Y^{**}$  which is separately $w^*$-continuous;
  \item $f^{\circledast\ast}(Y^{\ast}\times X_{1}^{\ast\ast}\times X_{2}^{\ast\ast})\subseteq X_{3}^{\ast}$ and
  $f^{\circledast\ast\ast}(X_{3}^{\ast\ast}\times Y^{\ast}\times X_{1}^{\ast\ast})\subseteq X_{2}^{\ast};$
  \item  For each $y^*\in Y^*$, the first Arens extension of the bilinear map $(\cdot,\cdot)\mapsto f^{\ast}(y^{\ast},\cdot,\cdot):X_1\times X_2\longrightarrow X_3^*$ is $X_3^*$-valued, and the operator $\cdot\mapsto f^{\ast\ast}(x_{3}^{\ast\ast},y^{\ast},\cdot): X_{1}\rightarrow X_{2}^{\ast}$ is weakly compact for every $y^*\in Y^*, x_{3}^{\ast\ast}\in X_{3}^{\ast\ast}$.
              \end{enumerate}
\end{theorem}
\begin{proof}
$(a)\Leftrightarrow (b)$ This  follows straightforwardly from the discussion in Remark \ref{r new}.\smallskip

$(b)\Rightarrow (c)$ Since $f$ is Aron--Berner regular the extension $f^\circledast$ is separately $w^*$-continuous. To prove  $f^{\circledast\ast}(Y^{\ast}\times X_{1}^{\ast\ast}\times X_{2}^{\ast\ast})\subseteq X_{3}^{\ast}$, it is enough to show that  for every  $y^{\ast}\in Y^{\ast}, x_{1}^{\ast\ast}\in X_{1}^{\ast\ast}, x_{2}^{\ast\ast}\in X_{2}^{\ast\ast}$, the linear functional $f^{\circledast\ast}(y^{\ast}, x_{1}^{\ast\ast}, x_{2}^{\ast\ast}): X_{3}^{\ast\ast}\longrightarrow\mathbb C$ is $w^{*}$-continuous. Take a net   $({x_3}_\alpha^{**})$  in $X_3^{\ast\ast} $, $w^*$-converging  to $x_{3}^{\ast\ast}\in X_{3}^{\ast\ast}$. Then
\begin{align}
\lim_{\alpha}\langle f^{\circledast\ast}(y^{\ast},x_{1}^{\ast\ast},x_{2}^{\ast\ast}), {x_3}_\alpha^{**}\rangle
=\lim_{\alpha}\langle f^{\circledast}(x_{1}^{\ast\ast},x_{2}^{\ast\ast},{x_3}_\alpha^{**}), y^{\ast}\rangle \nonumber
&=\langle f^{\circledast}(x_{1}^{\ast\ast},x_{2}^{\ast\ast},x_{3}^{\ast\ast}),y^{\ast}\rangle \nonumber\\
&=\langle f^{\circledast\ast}(y^{\ast},x_{1}^{\ast\ast},x_{2}^{\ast\ast}),x_{3}^{\ast\ast}\rangle. \nonumber
\end{align}
Similarly,  $f^{\circledast\ast\ast}(x_{3}^{\ast\ast}, y^{\ast}, x_{1}^{\ast\ast}): X_{2}^{\ast\ast}\longrightarrow\mathbb C$ is $w^{*}$-continuous for every $y^{\ast}\in Y^{\ast}, x_{1}^{\ast\ast}\in X_{1}^{\ast\ast}, x_{3}^{\ast\ast}\in X_{3}^{\ast\ast}$, which implies that $f^{\circledast\ast\ast}(X_{3}^{\ast\ast}\times Y^{\ast}\times X_{1}^{\ast\ast})\subseteq X_{2}^{\ast}.$\smallskip

$(c)\Leftrightarrow(d)$ If we  set $h:=f^{\ast}(y^{\ast},\cdot,\cdot):X_1\times X_2\longrightarrow X_3^*$ and $T:=f^{\ast\ast}(x_{3}^{\ast\ast},y^{\ast},\cdot):X_{1}\rightarrow X_{2}^{\ast}$, then a direct verification reveals that
\[h^{***}(x_1^{**},x_2^{**})=f^{\circledast\ast}(y^{\ast},x_1^{**},x_2^{**})\ \mbox{and}\ T^{**}(x_1^{**})=f^{\circledast\ast\ast}(x_{3}^{\ast\ast},y^{\ast},x_1^{**})\qquad (x_1^{**}\in X_1^{**},x_2^{**}\in X_2^{**}).\]
Now the conclusion follows from the fact that $(c)$ holds if and only if $h^{***}$ is $X_3^{*}$-valued and $T^{**}$ is $X_2^*$-valued, and the latter one is equivalent to the weak compactness of $T.$ \medskip

$(c)\Rightarrow(a)$ Let $({x_3}_\alpha)$ be a net in $X_3$, $w^*$-converging  to $x_{3}^{\ast\ast}$ in $X_3^{**}$. Then
\begin{align*}
\langle f^{\circledast}(x_{1}^{\ast\ast},x_{2}^{\ast\ast},x_{3}^{\ast\ast}),y^{\ast}\rangle
&=\langle f^{\circledast\ast}(y^{\ast},x_{1}^{\ast\ast},x_{2}^{\ast\ast}),x_{3}^{\ast\ast}\rangle\\
&=\lim_{\alpha}\langle f^{\circledast\ast}(y^{\ast},x_{1}^{\ast\ast},x_{2}^{\ast\ast}),{x_{3}}_{\alpha}\rangle\tag{$f^{\circledast\ast}(y^{\ast},x_{1}^{\ast\ast},x_{2}^{\ast\ast})\in X_3^*$}\\
&=\lim_{\alpha}\langle f^{\circledast}(x_{1}^{\ast\ast},x_{2}^{\ast\ast},{x_{3}}_{\alpha}),y^{\ast}\rangle\\
&=\lim_{\alpha}\langle f^{\sigma_{4}\circledast\sigma_{5}}(x_{1}^{\ast\ast},x_{2}^{\ast\ast},{x_{3}}_{\alpha}),y^{\ast}\rangle\tag{Lemma \ref{B9}(c)}\\
&=\langle f^{\sigma_{4}\circledast\sigma_{5}}(x_{1}^{\ast\ast},x_{2}^{\ast\ast},x_{3}^{\ast\ast}),y^{\ast}\rangle.
\end{align*}
We now  suppose that $({x_2}_\alpha)$ is a  net in $X_2$, $w^*$-converging to $x_{2}^{\ast\ast}$ in $X_2^{**}$. Then
\begin{align*}
\langle f^{\circledast}(x_{1}^{\ast\ast},x_{2}^{\ast\ast},x_{3}^{\ast\ast}),y^{\ast}\rangle
&=\langle f^{\circledast\ast}(y^{\ast},x_{1}^{\ast\ast},x_{2}^{\ast\ast}),x_{3}^{\ast\ast}\rangle\\
&=\langle f^{\circledast\ast\ast}(x_{3}^{\ast\ast},y^{\ast},x_{1}^{\ast\ast}),x_{2}^{\ast\ast}\rangle\\
&=\lim_{\alpha}\langle f^{\circledast\ast\ast}(x_{3}^{\ast\ast},y^{\ast},x_{1}^{\ast\ast}),{x_{2}}_{\alpha}\rangle\tag{$f^{\circledast\ast\ast}(x_{3}^{\ast\ast}, y^{\ast},x_{1}^{\ast\ast})\in X_2^*$}\\
&=\lim_{\alpha}\langle f^{\circledast}(x_{1}^{\ast\ast},{x_{2}}_{\alpha},x_{3}^{\ast\ast}),y^{\ast}\rangle\\
&=\lim_{\alpha}\langle f^{\sigma_{3}\circledast\sigma_{3}}(x_{1}^{\ast\ast},{x_{2}}_{\alpha},x_{3}^{\ast\ast}),y^{\ast}\rangle\tag{Lemma \ref{B9}(b)}\\
&=\langle f^{\sigma_{3}\circledast\sigma_{3}}(x_{1}^{\ast\ast},x_{2}^{\ast\ast},x_{3}^{\ast\ast}),y^{\ast}\rangle.
\end{align*}
We therefore get $f^{\circledast}=f^{\sigma_{3}\circledast\sigma_{3}}=f^{\sigma_{4}\circledast\sigma_{5}}$. Since the permutations in the triple $(\sigma_0,\sigma_3,\sigma_4)$ have different values at $1$, the mapping $f$ is Aron--Berner regular by Lemma \ref{c1} and this completes the proof.
\end{proof}

Theorem \ref{c6}  is actually a trilinear analogue of  \cite[Theorem 2.1]{MV}, where the Arens regularity of a bounded bilinear map is investigated. 
Wider information on Arens regularity of bilinear maps and Banach algebras can be found in \cite{Ar,D} (see, also, \cite{MV}).

\begin{remark}\label{c2.2} Let $A$ be a Banach algebra, and let $\Box$ and $\lozenge$ denote the first and second Arens products on $A^{**}$, respectively. According to \cite[Definition 2.17]{DL}, the (associative) topological centers of $A^{**}$ are $$\mathfrak{Z}^{(1)}_t (A^{**}) = \left\{ m\in A^{**} : m \Box \cdot = m\lozenge \cdot \ \right\};$$
$$\mathfrak{Z}^{(2)}_t (A^{**}) = \left\{ m\in A^{**} : \cdot \Box m = \cdot\lozenge m \ \right\}.$$ Both topological centers are norm-closed subalgebras of $(A^{**},\Box)$ and $(A^{**},\lozenge)$, both contain  $A$ itself. Furthermore $A$ is Arens regular if and only if either $\mathfrak{Z}^{(1)}_t (A^{**}) = A^{**}$ or $\mathfrak{Z}^{(2)}_t (A^{**}) = A^{**}$.\smallskip

Based on Lemma \ref{c1} and inspired by the above definition of topological centers for associative Banach algebras,  to explore the separate w$^*$-continuity of the extension $f^{\sigma\circledast\sigma ^{-1}}$ of a bounded trilinear mapping in more extent, we define the corresponding topological  centers. Let us fix $j\in \{1,2,3\}$. For each $3$-tuple $(\sigma,\tau,\rho) \in S_{3}$ attaining different values at $1$ with $\sigma(1) =j$, we define the $j$th $(\sigma,\tau,\rho)$-topological center $Z_{\sigma,\tau,\rho}^{j}(f)$ of a trilinear mapping $ f:X_{1} \times X_{2}\times X_{3} \longrightarrow Y$ as follows:
$$ Z_{\sigma,\tau,\rho}^j(f) =\left\{ x_{\sigma(1)}^{**}\in X_{\sigma(1)}^{**}:
f^{\sigma\circledast\sigma ^{-1}}(x_{\sigma(1)}^{**},\cdot,\cdot) = f^{\rho\circledast\rho^{-1}}(x_{\sigma(1)}^{**},\cdot,\cdot)
=f^{\tau\circledast\tau^{-1}}(x_{\sigma(1)}^{**},\cdot,\cdot)\right\}.$$
We observe that  $Z_{\sigma,\tau,\rho}^j(f) = Z_{\sigma,\rho,\tau}^j(f)$ for every $\sigma, \rho,\tau$ under the above conditions.
It is easy to verify that  $Z_{\sigma,\tau,\rho}^j (f)$ is a closed subspace of  $X_{\sigma(1)}^{**}= X_{j}^{**}$.\smallskip

We can conclude from Lemma \ref{c1} that,   for a $3$-tuple $(\sigma,\tau,\rho)$ satisfying the above assumptions, the equality $Z_{\sigma,\tau,\rho}^j (f)=X_{j}^{**}$ is equivalent to the Aron--Berner regularity of $f,$ and that, in this desirable case, we also have  $Z_{\sigma',\tau',\rho'}^k(f)=X_{k}^{**}$ and $Z_{\sigma'',\tau'',\rho''}^l(f)=X_{l}^{**}$ for every $k\neq l$ in $\{1,2,3\}\backslash\{j\}$, ${\sigma',\tau',\rho',\sigma'',\tau'',\rho''}\in S_3$ attaining different values at $2$ and at $3$, respectively, with  $\sigma' (1) = k$ and $\sigma''(1) = l$.  Similar notions of ``topological centers" for a bilinear mapping and for the product of an associative Banach algebra were extensively discussed in the literature of Banach algebras (see, for example, \cite{D,DL} and references therein). We shall return to the notion of topological center of certain triple maps in Example \ref{example new topological centers}, where we shall illustrate that the topological centers for trilinear maps exhibit a different behavior than the one we have for associative topological centers.
\end{remark}

Before proceeding with an applicable consequence of Theorem \ref{c6}, we recall that a bounded linear map $T:X\longrightarrow Y$ is  weakly compact if and only if the second adjoint, $T^{**},$ of $T$ is $Y$-valued,  that is, $T^{**}(X^{**})\subseteq Y.$ However, we do not have this equivalence   in the context of multilinear maps. If  a bounded multilinear map from the Cartesian product of some normed spaces to a normed space $Y$ is weakly compact, then, as it is shown in \cite[Proposition 4]{CGV},  all of its Aron--Berner extensions to the bidual spaces are $Y$-valued; however, there are many nonweakly compact multilinear maps with $Y$-valued Aron--Berner extensions. For instance, such as  presented in \cite[Example 4]{CGV}, the bounded bilinear map $(x,y)\mapsto h(x,y)=p(x)p(y):\ell^\infty\times\ell^\infty\longrightarrow\ell^1,$ where $p:\ell^\infty\longrightarrow\ell^2$ is an arbitrary  surjective bounded  linear map, is nonweakly compact, while its Aron--Berner extension (or actually, its Arens extension \cite{Ar}) $h^{***}$ is $\ell^1$-valued (see aso \cite{BomVi}). For further information concerning multilinear maps whose Aron--Berner extensions are $Y$-valued, the reader can consult \cite{AG, GV, PVMY}. To prove Theorem \ref{c7}, we need the following lemma.

\begin{lemma}\label{c61} If every bounded  linear map from $X_i, (i=1,2),$ into $Y$ is weakly compact, then the Arens extension{\rm(}s{\rm)} of every  bounded bilinear map from $X_1\times X_2$ into $Y$ is $Y$-valued.
\end{lemma}

\begin{proof} For a bounded bilinear map $h:X_1\times X_2\longrightarrow Y$, satisfying the hypotheses of the lemma, the bounded linear map $\cdot\mapsto T=h(x_1,\cdot):X_2\longrightarrow Y, (x_1\in X_1)$ is weakly compact, we get $h^{***}(x_1,x_2^{**})=T^{**}(x_2^{**})\in Y$ for each $x_2^{**}\in X_2^{**}.$  Again by the assumption, the bounded linear map $\cdot\mapsto S=h^{***}(\cdot,x_2^{**}):X_1\longrightarrow Y$ is weakly compact, which forces $h^{***}(x_1^{**},x_2^{**})=S^{**}(x_1^{**})\in Y,$ for each $x_1^{**}\in X_1^{**}$,  as claimed.
\end{proof}

The  preceding lemma together with  Theorem \ref{c6} lead us to rediscover a result due to Bombal and Villanueva (see \cite[Theorem 1]{BomVi}). The proof is slightly different here.

\begin{theorem}\label{c7} \cite[Theorem 1]{BomVi}
Let  $ f:X_{1} \times X_{2}\times X_{3} \longrightarrow Y$ be a bounded trilinear map. If for all $i\neq j$ every bounded linear map from $X_i$ into $X_j^*$ is weakly compact, then $f$ is Aron--Berner regular.
\end{theorem}

\begin{proof}
Since all bounded linear maps $X_1\longrightarrow X_3^*$ and $X_2\longrightarrow X_3^*$ are weakly compact, Lemma \ref{c61} implies that  the bilinear map  $(\cdot,\cdot)\mapsto f^{\ast}(y^{\ast},\cdot,\cdot):X_1\times X_2\longrightarrow X_3^*$ has $X_3^*$-valued Arens extension for every $y^*\in Y^*$. Further,  by hypothesis, the bounded linear map  $\cdot\mapsto f^{\ast\ast}(x_{3}^{\ast\ast},y^{\ast},\cdot): X_{1}\rightarrow X_{2}^{\ast}$  is also weakly compact for every fixed  $y^*\in Y, x_{3}^{\ast\ast}\in X_{3}^{\ast\ast}$. The conclusion now follows from Theorem \ref{c6}(d).
\end{proof}

\begin{remark}
One can compare Theorem \ref{c7} with the quoted result by Bombal and Villanueva. Theorem 1 in \cite{BomVi} proves the following: Let $X_{1},\ldots ,X_{k},$ and $X$ be Banach spaces satisfying that for each $i\ne j$ in $\{1,\ldots,k\}$ every bounded linear operator from $X_i$ into $X_j^{*}$ is weakly compact (this condition is satisfied whenever $X_1,\ldots , X_k$ have
property $(V)$ of Pelczy\'{n}ski, in particular when every $X_i$ is a C$^*$-algebra \cite[Corollary 6]{Pfi94} or a JB$^*$-triple \cite{ChuMe97}). Then every multilinear map $f:X_1\times\cdots\times X_k\to Y$ admits a (unique) separately $w^*$-continuous extension from $X_1^{**}\times\cdots\times X_k^{**}$ to  $Y^{**}$. Consequently, $f$ is Aron--Berner regular by Theorem \ref{c6}. A careful look, however, shows that the proof of Theorem \ref{c7} is a bit simpler. The same holds for multi-sesquilinear maps. In the case of JB$^*$-triples the hypotheses can be also derived from \cite[Lemma 5]{CIL} and in the setting of real JB$^*$-triples from \cite[Lemma 5]{PP}.\smallskip
\end{remark}

A Banach space $X$ is called regular if every bounded linear map from $X$ into $X^*$ is weakly compact (see \cite{L}). As we have mentioned in the above remark,  Banach spaces  having property $(V)$ of Pelczy\'{n}ski, in particular C$^*$-algebras and JB$^*$-triples, are the main examples of regular Banach spaces.\smallskip

As an immediate consequence of Theorem \ref{c7} we present the following result.
\begin{corollary}\label{regular space}Every bounded trilinear map on a regular normed  space is Aron--Berner regular. More precisely, if $X$ is regular,  then every bounded trilinear map $f:X\times X\times X\longrightarrow X$ is Aron--Berner regular.
\end{corollary}

\section{Biduals of Jordan Banach triple systems  and JB$^{\ast}$-triples}\label{second}

We first recall some basic notions and include some auxiliary results on Jordan Banach triple systems and JB$^*$-triples. Let $E$ be a complex (respectively, real) linear space. A complex (respectively, real) triple product on $E$ is a mapping
$$\pi:E \times E \times E \rightarrow E,\quad \pi(a,b,c)=\{a,b,c\} \quad (a,b,c\in E)$$
which is bilinear and symmetric in the outer variables and conjugate linear  (respectively, linear) in the middle one, satisfying the so-called Jordan identity
\begin{align}\label{Jordan Identity}
  \{a,b,\{c,d,e\}\} = \{\{a,b,c\},d,e\} - \{c,\{b,a,d\},e\} + \{c,d,\{a,b,e\}\}\qquad (a, b, c, d, e\in E).
\end{align} The pair $(E,\pi)$ is called a \emph{Jordan triple system}.\smallskip


If $E$ is a (real or complex) Banach space and the triple product is jointly continuous then the pair $(E,\pi)$ is called a \emph{Jordan Banach triple system}. A  JB$^*$-triple is a complex Jordan Banach triple system $E$ satisfying the following axioms:
\begin{enumerate}[$\bullet$]
  \item For any $a$ in $E$, the mapping $\cdot\mapsto \{a,a,\cdot\}$ is a Hermitian operator on $E$ with nonnegative spectrum;
  \item $\|\{a, a, a\}\| = \| a \|^3$ for all $a$ in $E$.
\end{enumerate}\smallskip
A Jordan Banach triple system $E$ is called abelian if the identity
  \begin{align}\label{abelian}
 \{a,b,\{c,d,e\}\}=\{\{a,b,c\},d,e\}\quad (a,b,c,d,e\in E),
 \end{align}
 is satisfied. In view of the Jordan identity \eqref{Jordan Identity}, however, the foregoing identity \eqref{abelian} is equivalent to the identity $\{a,b,\{c,d,e\}\}=\{a,\{b,c,d\},e\}.$\smallskip

A  subtriple of a Jordan Banach triple system $E$ is a subspace $F$ of $E$ satisfying $\{F, F, F\} \subseteq F$.\smallskip

For the sake of convenience, in the rest of the paper, we will be dealing with  complex Jordan Banach triple systems.\smallskip

For example, every Banach $^*$-algebra is a Jordan Banach triple system with respect to the triple product $\{a, b, c\} = \frac{1}{2}(ab^*c+cb^*a),$ and the same triple product equips every C$^*$-algebra and the Banach space $B(H,K),$ of all bounded linear operators between
two  Hilbert spaces $H$ and $K$, with a structure of JB$^*$-triple. A JB$^*$-algebra, with Jordan product $\circ$, is a  JB$^*$-triple under the triple product $\{a, b, c\} = (a \circ b^*) \circ c+(c \circ b^*) \circ a - (a \circ c) \circ b^*$, and a Hilbert space $H$ with inner product $\langle\cdot,\cdot\rangle$, is a JB$^*$-triple when endowed with the triple product  $\{a, b, c\} = \frac{1}{2}(\langle a,b\rangle c+\langle c,b\rangle a).$ \smallskip

A JB$^*$-triple which is a dual Banach space is called a JBW$^*$-triple. It is well known that the triple product of a JBW$^*$-triple is separately $w^*$-continuous \cite{BT} (this conclusion also holds for real JB$^*$-triples, see \cite{martinez2000separate}). 
We refer the readers to \cite{russo1994structure} and \cite{Ch} for the basic background  on Jordan Banach triple systems and JB$^*$-triples.

\subsection{Aron--Berner triple products on the bidual of Jordan Banach triple systems}\label{second1}
Let $E$ be a Jordan Banach triple system equipped with the triple product $\pi=\{\cdot,\cdot,\cdot\}:E\times E\times E\to E.$ Following the procedure described in subsection \ref{extension}, we can find six, generally different, norm preserving extensions of $\pi$ to $E^{**}\times E^{**}\times E^{**}$ defined by
\begin{align}\label{conju}
\pi^0=\pi^{*{\bar{*}}{\bar{*}}*}, \pi^1=\pi^{\sigma_1{\bar{*}}{\bar{*}}**\sigma_1}, \pi^2=\pi^{\sigma_2*{\bar{*}}{\bar{*}}*\sigma_2}, \pi^3=\pi^{\sigma_3**{\bar{*}}{\bar{*}}\sigma_3}, \pi^4=\pi^{\sigma_4{\bar{*}}{\bar{*}}**\sigma_5}, \mbox{and } \pi^5:=\pi^{\sigma_5**{\bar{*}}{\bar{*}}\sigma_4},
\end{align}
in which,  for the sake of conjugate linearity in the middle variable of each item, we used the conjugate adjoint ``$ \bar{*} $"  instead of  the usual adjoint
``$ * $" in some stages. Indeed, in the case when $f:X\times X\times X \longrightarrow X$ is a continuous map which is conjugate linear in the third variable, we define the conjugate transpose $f^{\bar{*}}:X^*\times X\times X\to X^*$ of $f$ by
\begin{align*}
 \langle f^{\bar{*}}(x^*,a,b),c\rangle=\overline{\langle x^*,f(a,b,c)\rangle}\quad  (a,b,c\in X, x^*\in X^*).
  \end{align*}
It is worthwhile mentioning that, taking conjugation does not affect the $w^*$-continuity of the extensions presented in \eqref{conju}, and they  enjoy the same $w^*$-continuity property as  described in Proposition \ref{B2}. We can consider by this procedure multi-sesquilinear maps.\smallskip

Following Definition \ref{B5}, we say that a Jordan Banach triple system $(E, \pi)$ is Aron--Berner regular if all the six extensions of $\pi$ presented in \eqref{conju} coincide on $E^{**}\times E^{**}\times E^{**}.$ In Example \ref{not} below we shall exhibit a Jordan Banach triple system which is not Aron--Berner regular; however, in the more favorable setting of JB$^*$-triples, a consequence of Corollary \ref{regular space} proves the next result.

\begin{corollary}\label{regularity}
Every JB$^*$-triple is Aron--Berner regular.
\end{corollary}

It should be remarked that, not only the triple product of a JB$^*$-triple  is Aron--Berner regular, but also every multi-sesquilinear map 
 on a JB$^*$-triple is Aron--Berner regular.\smallskip

Let $(E,\pi)$ be a Jordan Banach triple system. Our main objective in the rest of this section is to find sufficient conditions to guarantee that the bidual, $E^{**},$ of $E$ is a Jordan Banach triple system when it is equipped with some of the triple extensions $\pi^i$ for $i\in\{0,\ldots,5\}$. For this purpose,  we first note that, by construction, each of the assignments $\pi^i:E^{**}\times E^{**}\times E^{**}\to E^{**}$ in \eqref{conju} is bilinear in the outer variables and conjugate linear in the middle one. However, as our next example shows,  $\pi^i$ is not, in general, symmetric in the outer variables. 
The construction exhibited here is based on the well-known fact that the bidual of a commutative Banach algebra need not be, in general, commutative, and it is the case if and only if the original algebra is Arens regular (see \cite{Ar}).

\begin{example}\label{not}
Let $E$ be a commutative unital Banach $^*$-algebra which is not Arens regular (e.g., the convolution group algebra $(\ell^1(\mathbb Z),\ast)$, see \cite[Examples 2.6.22]{D}). Since $E$ is commutative, it is easy to check that $m\Box n=n\lozenge m,$ for all $m,n\in E^{**},$ where $\Box$ and $\lozenge$ stand for the first and second Arens products, respectively. Since $E$ is not Arens regular $\Box$ and $\lozenge$ are not commutative (see \cite{Ar} or \cite[comments after Corollary 2.6.18]{D}).\smallskip

Then $E$ can be viewed as a Jordan Banach triple system equipped with the triple product $\pi(a,b,c)=ab^*c.$ A direct verification, applying that $\Box$ and $\lozenge$ are associative (cf. \cite[Proposition A.3.53]{D}), reveals that for every $m,n,p\in E^{**}$ we have
\begin{equation*}
\begin{aligned}[c]
\pi^0(m,n,p)&=m\Box n^*\Box p,\\
\pi^1(m,n,p)&=m\Box (n^*\lozenge p),\\
\pi^3(m,n,p)&=(m\lozenge n^*)\Box p,
\end{aligned}
\ \ \ \ \ \ \ \ \ \
\begin{aligned}[c]
\pi^2(m,n,p)&=m\lozenge n^*\lozenge p,\\
\pi^4(m,n,p)&=(m\Box n^*)\lozenge p, \\
\pi^5(m,n,p)&=m\lozenge (n^* \Box p),
\end{aligned}
\end{equation*}
where $n^*$ stands for the involution of $n$ naturally induced by the involution of $E,$ via the assignments: $\langle n^*,\phi\rangle:=\overline{\langle n,\phi^*\rangle}$ and $\langle \phi^*,a\rangle:=\overline{\langle\phi,a^*\rangle}$ for $a\in E, \phi\in E^*.$ For example, let us take, via Goldstine's theorem, three nets $(m_{\alpha_1})$, $(n_{\alpha_2}^*)$ and  $(p_{\alpha_3})$ in $E$ converging in the $w^*$-topology of $E^{**}$ to $m,n,p\in E^{**},$ respectively. Since the mappings $\cdot \Box m$, $a\Box \cdot$, $m\lozenge \cdot$, and $\cdot\lozenge a$ are $w^*$-continuous for all $m\in E^{**}$, $a\in E,$ it can be easily seen that $$\pi^0 (m,n,p) = w^*\hbox{-}\lim_{\alpha_1} \lim_{\alpha_2} \lim_{\alpha_3}\  m_{\alpha_1} \Box n_{\alpha_2}^*\Box p_{\alpha_3}= w^*\hbox{-}\lim_{\alpha_1} \lim_{\alpha_2} \ m_{\alpha_1} \Box n_{\alpha_2}^*\Box p$$
$$= w^*\hbox{-}\lim_{\alpha_1} \ m_{\alpha_1} \Box n^*\Box p = m\Box n^* \Box p,$$

$$\pi^1 (m,n,p) = w^*\hbox{-}\lim_{\alpha_1} \lim_{\alpha_3} \lim_{\alpha_2}\  m_{\alpha_1} \Box (n_{\alpha_2}^*\lozenge p_{\alpha_3})= w^*\hbox{-}\lim_{\alpha_1} \lim_{\alpha_3} \ m_{\alpha_1} \Box (n^*\lozenge p_{\alpha_3})$$
$$= w^*\hbox{-}\lim_{\alpha_1} \ m_{\alpha_1} \Box (n^*\lozenge p) = m\Box (n^* \lozenge p),$$
and so on for the other identities.\smallskip

We thus conclude that
\begin{enumerate}[$(a)$]
\item $(E,\pi)$ is not Aron--Berner regular, indeed $\pi^0\neq\pi^2$. Namely, having in mind that $E$ is unital and $\Box$ is not commutative, we can find $m,p\in E^{**}$ with $\pi^0 (m,1,p) = m\Box p \neq p\Box m= m\lozenge p = \pi^2 (m,1,p)$;
\item $\pi^0$ does not obey the Jordan identity. To clarify this, we examine the Jordan identity \eqref{Jordan Identity} with   $m, n\in E^{**}$ instead of $a,b,$ and $1$ (=the identity of $E$) for the other components $c,d,e.$ Since $(x\Box y)^* = y^* \lozenge x^*$, the Jordan identity holds if and only if $m\Box n^*=n^*\Box m$, and this is not the case, because $E$ is not Arens regular as a Banach algebra. The same conclusion applies to $\pi^1,$ $\pi^2,$ $\pi^3$, $\pi^4$ and $\pi^5$;
\item Neither of $\pi^0$, $\pi^1,$ $\pi^2,$ $\pi^3$, $\pi^4$ nor $\pi^5$ is symmetric in the outer variables;
\item In this example $\pi^0 (m,n,p) = \pi^2 (p,n,m),$ $\pi^1 (m,n,p) = \pi^4 (p,n,m),$ and $\pi^3 (m,n,p) = \pi^5 (p,n,m)$ for all $m,n,p\in E^{**}$.
\end{enumerate}
\end{example}

\begin{example}\label{example new topological centers} For each Banach space $X$, let $\kappa_{_X}$ denote the canonical inclusion of $X$ into $X^{**}$. Let $A= K(c_0)$ denote the Banach algebra of all compact operators on the Banach space $c_0$ of all null sequences equipped with the composition product. Arguing as in \cite[Example 2.5]{LauUl96} or \cite[Example 6.2 and comments before]{DL}, we deduce that $A^{**} =B(\ell_{\infty}),$ and furthermore the first and second Arens products on $B(\ell_{\infty})$ are given by \begin{equation}\label{eq Arens extension new example} m\Box n = mn, \hbox{ and } m\lozenge n = \mathcal{Q}(m) n, \hbox{ for all } m,n\in B(\ell_{\infty})
 \end{equation} respectively, where $\mathcal{Q}$ is a contractive projection on $B(\ell_{\infty})$ defined in the following way: given $a\in B(\ell_{\infty})$  we set $$\mathcal{Q}(a) =\kappa_{\ell_1}^* a^{**} \kappa_{c_0}^{**}\in B(\ell_\infty).$$ It is known that $\mathcal{Q}(A^{**})=\{ z^* : z\in B(\ell_1) \}$, $\mathcal{Q}(Id_{\ell_{\infty}}) = Id_{\ell_{\infty}},$ $\ker(\mathcal{Q}) = \{ a\in B(\ell_{\infty}) : a \kappa_{c_0} = 0\}$, $$\mathcal{Q}(b) a = b a, \, \forall b\in B(\ell_{\infty}) \Longleftrightarrow a (\ell_\infty) \subseteq c_0,$$
$$ a b= \mathcal{Q}(a) b , \, \forall b\in B(\ell_{\infty}) \Longleftrightarrow \mathcal{Q}(a) =a, \hbox{ and }$$ (compare \cite[Example 6.2]{DL}) and from the associativity of $\lozenge$ we get $\mathcal{Q}(a \mathcal{Q}(b))= \mathcal{Q}(a) \mathcal{Q}(b)$.\smallskip

Let $f: A\times A\times A\to A$ be the trilinear mapping defined by $f(a,b,c) = a b c$. In this case $$Z^1_{\sigma_0,\sigma_2,\sigma_3} (f)  = \left\{ m\in B(\ell_{\infty}) :  f^{\sigma_0\circledast\sigma_0^{-1}}(m,\cdot,\cdot) = f^{\sigma_2\circledast\sigma_2^{-1}}(m,\cdot,\cdot)
=f^{\sigma_3\circledast\sigma_3^{-1}}(m,\cdot,\cdot) \right\},$$
$$Z^2_{\sigma_0,\sigma_1,\sigma_3} (f)  = \left\{ n\in B(\ell_{\infty}) :  f^{\sigma_0\circledast\sigma_0^{-1}}(\cdot,n,\cdot) = f^{\sigma_1\circledast\sigma_1^{-1}}(\cdot,n,\cdot)
=f^{\sigma_3\circledast\sigma_3^{-1}}(\cdot,n,\cdot) \right\},$$ and
$$Z^3_{\sigma_0,\sigma_1,\sigma_2} (f)  = \left\{ p\in B(\ell_{\infty}) :  f^{\sigma_0\circledast\sigma_0^{-1}}(\cdot,\cdot,p) = f^{\sigma_1\circledast\sigma_1^{-1}}(\cdot,\cdot,p)
=f^{\sigma_2\circledast\sigma_2^{-1}}(\cdot,\cdot,p) \right\}.$$
We deduce from the arguments in Example \ref{not} that
\begin{equation*}
\begin{aligned}[c]
f^{\sigma_0\circledast\sigma_0^{-1}}(m,n,p)&=m\Box n\Box p,\\
f^{\sigma_1\circledast\sigma_1^{-1}}(m,n,p)&=m\Box (n\lozenge p),
\end{aligned}
\ \ \ \ \ \ \ \ \ \
\begin{aligned}[c]
f^{\sigma_2\circledast\sigma_2^{-1}}(m,n,p)&=m\lozenge n\lozenge p,\\
f^{\sigma_3\circledast\sigma_3^{-1}}(m,n,p)&=(m\lozenge n)\Box p.\end{aligned}
\end{equation*}

Combining these expressions with \eqref{eq Arens extension new example} we derive that $m\in Z^1_{\sigma_0,\sigma_2,\sigma_3} (f) $ if and only if $$ m n p = \mathcal{Q}(\mathcal{Q}(m)n) p = \mathcal{Q}(m) n p,$$ for all $n,p\in B(\ell_\infty)$. Taking $n=1$ we get $m p = \mathcal{Q}(m) p $ for all $p$, and hence $\mathcal{Q}(m) =m$. Back to the previous identity we get $m \mathcal{Q}(n)p = \mathcal{Q}(m)\mathcal{Q}(n) p =   \mathcal{Q}(\mathcal{Q}(m)n) p = \mathcal{Q}(m) n p =m np$ for all $n,p$, which implies that $m \mathcal{Q}(n) = m n$ for all $n$. Taking any $x_0\in \ell_\infty$, $\Gamma\in \ell_{\infty}^*$ with $\Gamma|_{c_0} =0$, and setting $n = x_0\otimes \Gamma\in B(\ell_{\infty})$, we have $\mathcal{Q}(n) =0,$ and thus $m(x_0) =0,$ for all $x_0$. We have shown that $Z^1_{\sigma_0,\sigma_2,\sigma_3} (f)=\{0\}.$\smallskip

An element $n\in Z^2_{\sigma_0,\sigma_1,\sigma_3} (f) $ if and only if $ m n p = \mathcal{Q}(m) n p = m \mathcal{Q}(n) p,$  for all $m,p\in B(\ell_\infty)$. The case $p=1=m$ gives $\mathcal{Q}(n) =n$, while $p=1$ now implies $m n = \mathcal{Q}(m) n  = m \mathcal{Q}(n),$ equivalently, $n(\ell_\infty) \subseteq c_0$. It is not hard to see that $$Z^2_{\sigma_0,\sigma_1,\sigma_3} (f) = \{n\in B(\ell_{\infty}) : \mathcal{Q}(n) =n, \hbox{ and  } n(\ell_\infty) \subseteq c_0 \}= \kappa_{K(c_0)} (K(c_0)),$$ where the last equality is explicitly proved in \cite[$(6.10)$ in page 63]{DL}.\smallskip

Finally, $p \in Z^3_{\sigma_0,\sigma_1,\sigma_2} (f) $ if and only if $ m n p = m \mathcal{Q}(n) p = \mathcal{Q}(\mathcal{Q}(m)n) p =  \mathcal{Q}(m) \mathcal{Q}(n) p,$  for all $n,m\in B(\ell_\infty)$. Taking $n=1$ we have $m p =  \mathcal{Q}(m) p$ for all $m\in B(\ell_\infty)$, equivalently, $p(\ell_\infty)\subseteq c_0$. Then the previous identities simplify in the form $$ m (n p) = \mathcal{Q}(m) (\mathcal{Q}(n) p) = \mathcal{Q}(m) (n p), \hbox{ for all } m,n\in B(\ell_{\infty}),$$ which proves that $n p (\ell_\infty)\subseteq c_0$ for all $n\in B(\ell_{\infty})$, and thus $p=0$.\smallskip

We have shown that $Z^3_{\sigma_0,\sigma_1,\sigma_2} (f)= Z^2_{\sigma_0,\sigma_1,\sigma_3} (f)\neq Z^1_{\sigma_0,\sigma_2,\sigma_3} (f)$.\smallskip

For $x=(x_n)\in c_0$, we set $\overline{x} := (\overline{x_n})$, and we define an involution on $A= K(c_0)$ given by $a^{\sharp} (x) =\overline{a(\overline{x})}$ ($a\in A= K(c_0)$). We equip $A$ with a structure of Jordan Banach triple system under the triple product $$\pi(a,b,c) = \frac12 (a b^{\sharp} c + c b^{\sharp} a), \, (a,b,c\in A).$$ It is not hard to see that \begin{equation*}
\begin{aligned}[c]
\pi^{0}(m,n,p)&=\frac12 (m\Box n^\sharp\Box p + p\lozenge n^\sharp \lozenge p) ,\\
\pi^{1}(m,n,p)&=\frac12 (m\Box (n^\sharp \lozenge p) + (p\Box n^\sharp)\lozenge m ) ,
\end{aligned}
\ \ \ \ \ \ \ \ \ \
\begin{aligned}[c]
\pi^{2}(m,n,p)&=\frac12( m\lozenge n^{\sharp}\lozenge p + p \Box n^{\sharp} m),\\
\pi^{3}(m,n,p)&=\frac12( (m\lozenge n^\sharp)\Box p + p \lozenge (n^\sharp\Box m) ).\end{aligned}
\end{equation*}
The arguments in the previous paragraphs can be applied to deduce that $$Z^1_{\sigma_0,\sigma_2,\sigma_3} (\pi)  = \left\{ m\in B(\ell_{\infty}) :  \pi^{0}(m,\cdot,\cdot) = \pi^{2}(m,\cdot,\cdot)
=\pi^{3}(m,\cdot,\cdot) \right\}=\{0\},$$ and
$$Z^2_{\sigma_0,\sigma_1,\sigma_3} (\pi)  = \left\{ n\in B(\ell_{\infty}) :  \pi^{0}(\cdot,n,\cdot) = \pi^{1}(\cdot,n,\cdot)
=\pi^{3}(\cdot,n,\cdot) \right\} $$ $$= \{n\in B(\ell_{\infty}) : \mathcal{Q}(n^{\sharp}) =n^{\sharp}, \hbox{ and  } n^{\sharp}(\ell_\infty) \subseteq c_0 \}= \kappa_{K(c_0)} (K(c_0)).$$

\medskip

We have seen in the previous example that contrary to the usual properties of the associative topological centers, the first topological center of a trilinear mapping  (or of a triple product) may be zero. In this extreme case we say that $f$ is ``strongly Aron--Berner irregular of type $1$". In particular, for the Jordan Banach triple system $(A,\pi)$ defined at the end of the previous example $Z^1_{\sigma_0,\sigma_2,\sigma_3} (\pi) $ need not contain $A$. To avoid this strange behaviour, one feels tempted to define topological centers for trilinear maps via Lemma \ref{B9} instead of Lemma \ref{c1} as principal motivation. For $\sigma\in S_3$, the alternative $\sigma$-topological centers of a trilinear operator $f:X_1\times X_2\times X_3\to Y$ are defined by 
$$\widetilde{Z}_{\sigma}^1(f) = \left\{ x_{\sigma(1)}^{**}\in X_{\sigma(1)}^{**}:\, \begin{array}{c}
                                                                                      f^{\sigma\circledast\sigma ^{-1}}(x_{\sigma(1)}^{**},\cdot,x_{\sigma(3)}^{**})\ \mbox{and}\ f^{\sigma\circledast\sigma ^{-1}}(x_{\sigma(1)}^{**},x_{\sigma(2)},\cdot) \mbox{ are } \\
                                                                                       w^*\mbox{-continuous for every}\  x_{\sigma(2)} \in X_{\sigma(2)}\ \mbox{and}\ x_{\sigma(3)}^{**}\in X_{\sigma(3)}^{**}
                                                                                    \end{array}\right\},$$
$$\widetilde{Z}_{\sigma}^2(f) = \left\{ x_{\sigma(2)}^{**}\in X_{\sigma(2)}^{**}:\, \begin{array}{c}
                                                                                      f^{\sigma\circledast\sigma ^{-1}}(\cdot, x_{\sigma(2)}^{**},x_{\sigma(3)}^{**})\ \mbox{and}\ f^{\sigma\circledast\sigma ^{-1}}(x_{\sigma(1)},x^{**}_{\sigma(2)},\cdot)\ \mbox{are} \\
                                                                                       w^*\mbox{-continuous for every}\  x_{\sigma(1)} \in X_{\sigma(1)}\ \mbox{and}\ x_{\sigma(3)}^{**}\in X_{\sigma(3)}^{**}
                                                                                    \end{array}\right\},$$
$$\widetilde{Z}_{\sigma}^3(f) = \left\{ x_{\sigma(3)}^{**}\in X_{\sigma(3)}^{**}:\, \begin{array}{c}
                                                                                      f^{\sigma\circledast\sigma ^{-1}}(\cdot, x_{\sigma(2)}^{**},x_{\sigma(3)}^{**})\ \mbox{and}\ f^{\sigma\circledast\sigma ^{-1}}(x_{\sigma(1)},\cdot,x^{**}_{\sigma(3)})\ \mbox{are} \\
                                                                                        w^*\mbox{-continuous for every}\  x_{\sigma(1)} \in X_{\sigma(1)}\ \mbox{and}\ x_{\sigma(2)}^{**}\in X_{\sigma(2)}^{**}
                                                                                    \end{array}\right\}.$$

It is easy to verify that  $\widetilde{Z}_{\sigma}^j(f)$ is a closed subspace of  $X_{\sigma(j)}^{**}$ containing  $X_{\sigma(j)}$ (cf. Lemma \ref{B9}). However these alternative topological centers are not conclusive to measure the Aron--Berner regularity of $f$. For example, for $\sigma = \sigma_0$ and the mapping $f$ in the first part of this example we have $$\widetilde{Z}_{\sigma_0}^3(f)=\Big\{ p\in A^{**}:\, f^{\circledast} (\cdot, n,p)\ \mbox{and}\ f^{\circledast}(a,\cdot,p)\ \mbox{are } w^*\mbox{-continuous for every}\  a \in A\ \mbox{and}\ n\in A^{**} \Big\}
$$
$$=\Big\{ p\in A^{**}:\, \cdot\Box n\Box p\ \mbox{and}\ a\Box \cdot\Box p\ \mbox{are } w^*\mbox{-continuous for every}\  a \in A\ \mbox{and}\ n\in A^{**} \Big\} = A^{**},$$ but $f^{\circledast}$ is not separately weak$^*$-continuous.
\end{example}
\smallskip

The, a priori different, norm preserving Arens extensions (i.e. $\Box$ and $\lozenge$) associated with the binary product of an associative Banach algebra $A$ are always associative (see \cite[Proposition A.3.53]{D}). That is, the associative property lifts from $A$ to $(A^{**},\Box)$ and $(A^{**},\lozenge)$. In Jordan triple systems the role of associativity is somehow played by the Jordan identity. Let $(E,\pi)$ be a Jordan Banach triple system. We are naturally led to ask whether $(E^{**},\pi^i)$ satisfies the Jordan identity for some (all) $i\in \{0,\ldots,5\}$. We have seen in the previous example that, in general, the Jordan identity does not lift from $(E,\pi)$ to $(E^{**},\pi^i)$.

\begin{proposition}\label{l AB regularity implies the Jordan identity in the second dual} Let $(E,\pi)$ be a Jordan Banach triple system. If $(E,\pi)$ is Aron--Berner regular. Then  $(E^{**},\pi^i)$ is a Jordan Banach triple system for each $i\in \{0,1,\ldots,5\}.$
\end{proposition}

\begin{proof} It follows from Theorem \ref{c6} and Remark \ref{r new} that for each $i\in \{0,1,\ldots,5\}$ the extension $\pi^i: E^{**}\times E^{**}\times E^{**}\to E^{**}$ is separately weak$^*$-continuous. Since $\pi^i|_{E^3} = \pi$ and the latter satisfies the Jordan identity \eqref{Jordan Identity}, an iterated process taking consecutive weak$^*$-limits on the variables proves that $\pi^i$ satisfies the Jordan identity. The same argument implies that $\pi^i$ is symmetric in the outer variables, and hence $(E^{**},\pi^i)$ is a Jordan Banach triple system.
\end{proof}

Combining the above proposition with Corollary \ref{regularity}, one can deduce that the bidual, $E^{**}$, of a JB$^*$-triple  $(E,\pi)$ is  Jordan Banach triple system under each Aron--Berner triple product. The celebrated Dineen's theorem  (see \cite{dineen2}) actually proves a stronger conclusion by showing that $E^{**}$ is a JB$^*$-triple under an appropriate triple product. A finer analysis developed by Barton and Timoney concludes that $E^{**}$ is a JB$^*$-triple whose triple product (denoted by $\widetilde{\pi}$) extends that on $E$ and is separately weak$^*$-continuous (see \cite[Theorem 1.4]{BT}). The arguments in Remark \ref{r new} imply that $\widetilde{\pi}=\pi^i$ for all $i=0,\ldots,5$, consequently, we have the following fact recasting Dineen-Barton-Timoney theorem in the language of Aron--Berner triple products.

\begin{theorem}\label{e13}
Every JB$^{\ast}$-triple $E$ is Aron--Berner regular and its bidual, $E^{**},$ equipped with the triple product provided by the unique Aron--Berner extension of the triple product of $E$, is a JB$^*$-triple.
\end{theorem}

Although a particular extension $\pi^i$ need not be symmetric in the outer variables, the Aron--Berner extensions enjoy certain symmetry whenever we mix the six different extensions. We shall show next that the identities exhibited in Example \ref{not}$(d)$ are a pattern satisfied by the triple product of every Jordan Banach triple system.

\begin{proposition}\label{symmetry} Let $(E,\pi)$ be a Jordan Banach triple system. Then for every $m,n,p\in E^{**},$ we have
\begin{align}\label{jab}
  \pi^0(m,n,p)=\pi^2(p,n,m), \quad \pi^1(m,n,p)=\pi^4(p,n,m),  \mbox{ and }\ \pi^3(m,n,p)=\pi^5(p,n,m).
 \end{align}
Furthermore, if $\pi^0$ is symmetric in the outer variables then $\pi^0=\pi^1=\pi^2=\pi^4$.
\end{proposition}

\begin{proof} Let $(m_{\alpha_1}), (n_{\alpha_2}),$ and $(p_{\alpha_3})$ be bounded nets in $E,$  $w^*$-converging to $m, n,$ and $p$, respectively.\smallskip

Using \eqref{itrated} or Proposition \ref{B2}, we have
\begin{align*}
 \pi^0(m,n,p)&=w^*\hbox{-}\lim_{\alpha_1} \lim_{\alpha_2} \lim_{\alpha_3} \pi(m_{\alpha_1},n_{\alpha_2}, p_{\alpha_3})\\
\hspace*{3.5cm}&=w^*\hbox{-}\lim_{\alpha_1} \lim_{\alpha_2} \lim_{\alpha_3} \pi(p_{\alpha_3}, n_{\alpha_2}, m_{\alpha_1})\\
                    &=w^*\hbox{-}\lim_{\alpha_1} \lim_{\alpha_2} \lim_{\alpha_3} \pi^{\sigma_2}(m_{\alpha_1},n_{\alpha_2}, p_{\alpha_3})\\
                    &=\pi^{\sigma_2*{\bar{*}}{\bar{*}}*}(m,n,p)\\
                    &=\pi^2(p,n,m).
                    \end{align*}
The remaining equalities follow by similar arguments.\smallskip

Suppose now that $\pi^0$ is symmetric in the outer variables; then by the previous part, $\pi^0=\pi^2.$
On the other hand, since by Lemma \ref{B9}$(c)$, $\pi^0(p,n,m_\alpha)=\pi^4(p,n,m_\alpha)$ for each $\alpha$,  we have
 \begin{align*}
 \pi^0(m,n,p)
 =w^*\hbox{-}\lim_\alpha \pi^0(m_\alpha,n,p)
                    =w^*\hbox{-}\lim_\alpha \pi^0(p,n,m_\alpha)
                     =w^*\hbox{-}\lim_\alpha \pi^4(p,n,m_\alpha)
                    =\pi^4(p,n,m),
                    \end{align*}\medskip
and the second identity in \eqref{jab} implies that $\pi^0=\pi^1.$ In particular, $\pi^1$ is  symmetric in the outer variables and  so is equal to $\pi^4.$ We have therefore shown that $\pi^0=\pi^1=\pi^2=\pi^4$.
\end{proof}

Let $(E,\pi)$ be a Jordan Banach triple system. We shall next explore when $(E^{**},\pi^i)$ satisfies the Jordan identity \eqref{Jordan Identity}.

%

\begin{theorem}\label{t outer symmetric is sufficient and necessary for pi0} Let $(E,\pi)$ be a Jordan Banach triple system. Then the extension  $\pi^0: E^{**}\times E^{**}\times E^{**}\to E^{**}$ {\rm(}equivalently, $\pi^2${\rm)} is symmetric in the outer variables if and only if $(E^{**}, \pi^0)$ {\rm(}equivalently, $(E^{**}, \pi^2)${\rm)} is a Jordan Banach triple system.
\end{theorem}

\begin{proof} The if implication is clear; let us assume that the extension $\pi^0$ is symmetric in the outer variables. It follows from Proposition \ref{symmetry} that $\pi^0=\pi^2$. We shall prove that $\pi^0$ satisfies the Jordan identity.\smallskip

The basic properties of the Aron--Berner extensions assure that the mappings $\pi^0 (\cdot,n,p)$, $\pi^0 (a,\cdot,p)$ and $\pi^0 (a,b,\cdot)$ are weak$^*$-continuous for all $n,p\in E^{**}$, $a,b\in E$ (cf. Proposition \ref{B2}). Since $\pi^0$ is symmetric in the outer variables, we also deduce that the mappings $\pi^0 (p,n,\cdot)$ and $\pi^0 (p,\cdot,a)$ are weak$^*$-continuous for all $n,p\in E^{**}$, $a\in E$.\smallskip

Let us pick five points $m,n,x,y,z\in E^{**}$ and five bounded nets $(m_{\alpha_1})$, $(n_{\alpha_2}),$ $(x_{\alpha_3}),$ $(y_{\alpha_4})$, and $(z_{\alpha_5})$ in $E$ converging in the $w^*$ topology of $E^{**}$ to $m,n,x,y,$ and $z$, respectively. By the Jordan identity in $E$ we have $$\pi(m_{\alpha_1},n_{\alpha_2},\pi(x_{\alpha_3},y_{\alpha_4},z_{\alpha_5})) = \pi(\pi(m_{\alpha_1},n_{\alpha_2},x_{\alpha_3}),y_{\alpha_4},z_{\alpha_5}) - \pi(x_{\alpha_3},\pi(n_{\alpha_2},m_{\alpha_1},y_{\alpha_4}),z_{\alpha_5}) $$ $$+ \pi(x_{\alpha_3},y_{\alpha_4},\pi(m_{\alpha_1},n_{\alpha_2},z_{\alpha_5})),$$ for all ${\alpha_1},{\alpha_2},{\alpha_3},{\alpha_4}$ and ${\alpha_5}$. Since the maps $\pi^0(\cdot,n_{\alpha_2},\pi^0(x_{\alpha_3},y_{\alpha_4},z_{\alpha_5}))$, $ \pi^0(\pi^0(\cdot,n_{\alpha_2},x_{\alpha_3}),y_{\alpha_4},z_{\alpha_5}),$ $\pi^0(x_{\alpha_3},\pi^0(n_{\alpha_2},\cdot,y_{\alpha_4}),z_{\alpha_5}),$ and $\pi^0(x_{\alpha_3},y_{\alpha_4},\pi(\cdot,n_{\alpha_2},z_{\alpha_5}))$ are weak$^*$-continuous, by taking $w^*$-limits in $\alpha_1$ we get $$\pi^0(m,n_{\alpha_2},\pi(x_{\alpha_3},y_{\alpha_4},z_{\alpha_5})) = \pi^0(\pi^0(m,n_{\alpha_2},x_{\alpha_3}),y_{\alpha_4},z_{\alpha_5}) - \pi^0(x_{\alpha_3},\pi^0(n_{\alpha_2},m,y_{\alpha_4}),z_{\alpha_5}) $$ $$+ \pi^0(x_{\alpha_3},y_{\alpha_4},\pi^0(m,n_{\alpha_2},z_{\alpha_5})),$$ for all ${\alpha_2},{\alpha_3},{\alpha_4},$ and ${\alpha_5}.$ Having in mind that $\pi^0(m,\cdot,\pi(x_{\alpha_3},y_{\alpha_4},z_{\alpha_5})),$ $ \pi^0(\pi^0(m,\cdot,x_{\alpha_3}),y_{\alpha_4},z_{\alpha_5})$, $\pi^0(x_{\alpha_3},\pi^0(\cdot,m,y_{\alpha_4}),z_{\alpha_5}),$ and $ \pi^0(x_{\alpha_3},y_{\alpha_4},\pi^0(m,\cdot,z_{\alpha_5}))$ are weak$^*$-continuous, by taking $w^*$-limits in ${\alpha_2}$ we deduce that $$\pi^0(m,n,\pi(x_{\alpha_3},y_{\alpha_4},z_{\alpha_5})) = \pi^0(\pi^0(m,n,x_{\alpha_3}),y_{\alpha_4},z_{\alpha_5}) - \pi^0(x_{\alpha_3},\pi^0(n,m,y_{\alpha_4}),z_{\alpha_5}) $$ $$+ \pi^0(x_{\alpha_3},y_{\alpha_4},\pi^0(m,n,z_{\alpha_5})),$$ for all ${\alpha_3},{\alpha_4},$ and ${\alpha_5}.$ Now, applying that $\pi^0(m,n,\pi^0(x_{\alpha_3},\cdot,z_{\alpha_5})),$ $\pi^0(\pi^0(m,n,x_{\alpha_3}),\cdot,z_{\alpha_5}),$\linebreak $\pi^0(x_{\alpha_3},\pi^0(n,m,\cdot),z_{\alpha_5}),$ and $\pi^0(x_{\alpha_3},\cdot,\pi^0(m,n,z_{\alpha_5}))$ are weak$^*$-continuous and taking $w^*$-limits in ${\alpha_4}$ we obtain $$\pi^0(m,n,\pi^0(x_{\alpha_3},y,z_{\alpha_5})) = \pi^0(\pi^0(m,n,x_{\alpha_3}),y,z_{\alpha_5}) - \pi^0(x_{\alpha_3},\pi^0(n,m,y),z_{\alpha_5})$$ $$+ \pi^0(x_{\alpha_3},y,\pi^0(m,n,z_{\alpha_5})),$$ for all ${\alpha_3},$ and ${\alpha_5}.$\smallskip

In the penultimate step we observe that the maps $\pi^0(m,n,\pi^0(x_{\alpha_3},y,\cdot))$, $\pi^0(\pi^0(m,n,x_{\alpha_3}),y,\cdot)$, $\pi^0(x_{\alpha_3},\pi^0(n,m,y),\cdot),$ and $\pi^0(x_{\alpha_3},y,\pi^0(m,n,\cdot))$ are weak$^*$-continuous, so by taking $w^*$-limits in ${\alpha_5}$ in the latest identity we derive at $$\pi^0(m,n,\pi^0(x_{\alpha_3},y,z)) = \pi^0(\pi^0(m,n,x_{\alpha_3}),y,z) - \pi^0(x_{\alpha_3},\pi^0(n,m,y),z)+ \pi^0(x_{\alpha_3},y,\pi^0(m,n,z)),$$ for all ${\alpha_3}.$ Finally, taking weak$^*$-limits in $\alpha_3$ we have $$\pi^0(m,n,\pi^0(x,y,z)) = \pi^0(\pi^0(m,n,x),y,z) - \pi^0(x,\pi^0(n,m,y),z)+ \pi^0(x,y,\pi^0(m,n,z)),$$ witnessing that $\pi^0$ satisfies the Jordan identity.
\end{proof}

Our next example illustrates the optimality of Theorem \ref{t outer symmetric is sufficient and necessary for pi0}.

\begin{example}\label{example pi1 symmetric but not Jordan identity} As in Example \ref{not}, let $E$ be a commutative complex Banach $^*$-algebra which is not Arens regular, and let $\Box$ and $\lozenge$ denote the different Arens extensions of the product of $E$. Let us consider the triple product on $E$ defined by $\pi (x,y,z) = \phi (x y^*) z + \phi (z y^*) x$ ($x,y,z\in E$), where $\phi$ is a functional in $E^*$ satisfying $\phi (x^*) = \overline{\phi (x)}$, for all $x\in E$. It is not hard to check that, by applying \eqref{itrated} or Proposition \ref{B2}, we have
\begin{equation*}
\begin{aligned}[c]
\pi^0(m,n,p)&=\phi(m\Box n^*) p +\phi(p\lozenge n^*) m,\\
\pi^1(m,n,p)&=\phi(m\Box n^*) p +\phi(p\Box n^*) m,\\
\pi^3(m,n,p)&=\phi(m\lozenge n^*) p +\phi(p\lozenge n^*) m,
\end{aligned}
\ \ \ \ \ \ \ \ \ \
\begin{aligned}[c]
\pi^2(m,n,p)&=\phi(m\lozenge n^*) p +\phi(p\Box n^*) m,\\
\pi^4(m,n,p)&=\phi(m\Box n^*) p +\phi(p\Box n^*) m,\\
\pi^5(m,n,p)&=\phi(m\lozenge n^*) p +\phi(p\lozenge n^*) m,
\end{aligned}
\end{equation*} for all $m,n,p\in E^{**}$. Clearly, the extensions $\pi^0$ and $\pi^2$ are not, in general, symmetric in the outer variables, while $\pi^1=\pi^4$ and $\pi^3=\pi^5$ are symmetric in the outer variables. It is a bit laborious to check that $\pi^1$ and $\pi^3$ do not satisfy the Jordan identity. For example, the identity $$\pi^1(m,n,\pi^1(1,1,1)) = \pi^1(\pi^1(m,n,1),1,1)- \pi^1(1,\pi^1(n,m,1),1) + \pi^1(1,1,\pi^1(m,n,1))$$ holds if and only if $\phi (m\Box n^*) = \phi (n^*\Box m)$,  and the latter identity does not hold, in general, as $E$ is not Arens regular.
\end{example}

\begin{remark}
\begin{enumerate}[$(i)$]
\item  By mimicking the proof of the fact that the bidual of a C$^*$-algebra is itself a C$^*$-algebra  (see \cite[Section III.2]{T} and \cite[Theorem 3.2.36]{D}), one may feel tempted to provide the same direct proof for Theorem \ref{e13} but not based on the results proved by Dineen, Barton and Timoney which were built on holomorphic theory and ultraproducts techniques. This opens the door to an argument not based on holomorphic theory.

\item One can also extend the triple product of a JB$^*$-triple $(E,\pi)$ to the $(2n)$-dual, $E^{(2n)}$, where $\pi^{[n]}$ is inductively defined by the formulae $\pi^{[1]}:=\pi^0,\  \pi^{[n+1]}:={\big(\pi^{[n]}\big)}^0.$ Then Theorem \ref{e13} implies that
 $(E^{(2n)},\pi^{[n]})$ is a again a JB$^*$-triple for each $n\in\mathbb{N}$ and the triple product $\pi^{[n]}$ is separately $w^*$-continuous.
Niazi, Miri, and the second named author \cite{NME} used the extension $\pi^{[n]}$ and some related module operations for investigating the $n$-weak amenability of the bidual of a JB$^*$-triple.
\end{enumerate}
\end{remark}

\subsection{Aron--Berner triple products  versus ultrapower triple products}\label{second2}

We have already commented that, beside the Aron--Berner extensions (see \eqref{conju}), there is an alternative method to extend the triple product of  a Jordan Banach triple system $(E,\pi)$ its bidual, which relies on the so-called principle of local reflexivity  \cite[Proposition 6.6]{He}, based on an analysis of finite-dimensional subspaces of the second dual of a Banach space.  This result implies the existence of an ultrafilter $\mathcal U$ such that the bidual $E^{**}$ can be isometrically embedded into  the  ultrapower $E_{\mathcal U}$ of $E$, via a map  $J:E^{**}\longrightarrow E_{\mathcal U}$ such that $J(E^{**})$ is the range of a contractive projection  on $E_{\mathcal U}$ and the restriction of $J$ to $E$ is the canonical embedding of $E$ into $E_{\mathcal U}$ (see \cite[Proposition 5]{dineen2} and \cite[Proposition 6.7]{He}).\smallskip

More concretely, let $\mathcal{U}$ be an ultrafilter on a nonempty set $I$, and let $\{X_i \}_{i\in I}$ be a family of Banach spaces. The symbol $\ell_{\infty} (I,X_i) = \ell_{\infty} (X_i)$ will stand for the Banach space obtained as the $\ell_{\infty}$-sum of the family $\{X_i \}_{i\in I}$, while $${c}_0 \left(X_i\right) := \left\{ (x_i) \in \ell_{\infty} (X_i) : \lim_{\mathcal{U}} \|x_i\| =0 \right\}.$$
The \emph{ultraproduct} of the family $\{X_i \}_{i\in I}$ relative to the ultrafilter $\mathcal{U}$, denoted by $(X_i)_\mathcal{U}$,
is defined as the quotient Banach space $\ell_{\infty} (X_i)/{c}_0 \left(X_i\right)$ equipped with the quotient norm. Given an element $[x_i]_{\mathcal{U}}$ in $(X_i)_\mathcal{U}$ it is known that $$\| [x_i]_{\mathcal{U}} \| = \lim_{\mathcal{U}} \|x_i\|,$$ independently of the representative of $[x_i]_{\mathcal{U}}$. If all the spaces in the family $\{X_i \}_{i\in I}$ coincide with a Banach space $E$, then we speak of an \emph{ultrapower},
denoted by $E_{\mathcal{U}}$. Every Banach space $E$ can be isometrically embedded into its ultrapower via the mapping $J: E\to E_{\mathcal{U}},$ $J(x) = [x_i]_{\mathcal{U}}$, with $x_i=x$ for all $i\in I$.\smallskip

If we assume that $E$ is a Jordan Banach triple system with respect to the triple product $\pi$, then $\ell_{\infty} (E_i)$ and the ultrapower $E_{\mathcal{U}}$ become Jordan Banach triple systems with respect to pointwise triple product and
$$\pi_{\mathcal{U}} ([x_i]_{\mathcal{U}},[y_i]_{\mathcal{U}},[z_i]_{\mathcal{U}}) := [\pi(x_i,y_i,z_i)]_{\mathcal{U}}, \ \ ([x_i]_{\mathcal{U}},[y_i]_{\mathcal{U}},[z_i]_{\mathcal{U}}\in E_{\mathcal{U}}),$$ respectively. If we further assume that $\{E_i:i\in I\}$ is a family of JB$^*$-triples, the Banach space $\ell_{\infty} (E_i)$ is a JB$^*$-triple with pointwise triple product (see \cite[p. 523]{K} or \cite[Ex.\ 3.1.4]{Ch}), and $\mathcal{I}={c}_0 \left(E_i\right)$ is a closed triple ideal of $\ell_{\infty} (E_i)$ (i.e. $\{ E,E,\mathcal{I}\} + \{ E,\mathcal{I},E\}\subseteq \mathcal{I}$). Therefore, the quotient $(E_i)_{\mathcal{U}}= \ell_{\infty} (E_i)/{c}_0 \left(E_i\right)$ is a JB$^*$-triple (see \cite{K} or \cite[Corollary\ 3.1.18]{Ch}).\smallskip

One of the key results in the theory of ultrapowers is the following result due to Heinrich (see \cite[Proposition 6.7]{He}): For each Banach space $E$ there exist an ultrafilter $\mathcal{U}$ and an isometric embedding $J$ of $E^{**}$ into $E_{\mathcal{U}}$ satisfying:\begin{enumerate}[$(a)$]\item  The restriction of $J$ to $E$ is the canonical embedding of $E$ into $E_{\mathcal{U}}$;
\item If $J(a) = [a_i]_{\mathcal{U}}$ with $a\in E^{**}$, then $w^*$-$\lim_{\mathcal{U}} (a_i) = a$ and the mapping $P([x_i]_{\mathcal{U}}) = w^*$-$\lim_{\mathcal{U}} x_i$ ($[x_i]_{\mathcal{U}}\in E_{\mathcal{U}}$) is a contractive projection from $E_{\mathcal{U}}$ onto $J(E^{**})$.
\end{enumerate} Henceforth, we fix the ultrafilter $\mathcal{U}$ satisfying the above properties.\smallskip

Heinrich's result offers a procedure to project properties from $E_{\mathcal{U}}$ onto $E^{**}$. For example, if $(E,\pi)$ is a Jordan Banach triple system we can define a triple product on $E^{**}$ by the assignment \begin{align}\label{u}
\pi^{\mathcal U}(m,n,p)=w^*\hbox{-}\lim_{\mathcal U}\pi(a_\alpha, b_\alpha, c_\alpha)\quad (m,n,p\in E^{**}),
\end{align}  for all $J(m)=[a_\alpha]_{\mathcal U}, J(n)=[b_\alpha]_{\mathcal U}$ and $J(p)=[c_\alpha]_{\mathcal U}$ in $E^{**}\subseteq E_{\mathcal{U}}$,  where the convergence is guaranteed by the weak$^*$-compactness of the closed unit ball of $E^{**}$. We obtain this way an extension $\pi^{\mathcal U}$ of $\pi$ to $E^{**}\times E^{**}\times E^{**}.$ A similar procedure was explored by Iochum and Loupias in \cite{IL0} (see also \cite{IL}), where they assigned an ultrapower product to the bidual of a Banach algebra and compared it with the so-called Arens products \cite{Ar}.\smallskip

The extension $\pi^{\mathcal U}$ naturally reflects the bilinearity in the outer variables and conjugate linearity in the middle variable of $\pi$. It also reflects the symmetry in the outer variables of $\pi$ more naturally than the Aron--Berner extensions (cf. Example \ref{not}). But in general, $\pi^{\mathcal U}$ does not satisfy the Jordan  identity  \eqref{Jordan Identity}. If we further assume that $\pi^{\mathcal U}$ is separately $w^*$-continuous, then, by Remark \ref{r new}, $\pi^{\mathcal U}=\pi^i$ for each $i=0,\ldots,5.$ In particular, $(E,\pi)$ is Aron--Berner regular and  $(E^{**},\pi^{\mathcal U})$ is a Jordan Banach triple system.\smallskip

When $E$ is a JB$^*$-triple, the \emph{contractive projection principle} (see \cite{Ka84} and \cite{Sta82}) is the key tool applied by Dineen and Barton-Timoney to equip $E^{**}$ with an optimal triple product. In the wider setting of Jordan Banach triple systems the contractive projection principle is simply hopeless.\smallskip

In the next result, we compare the ultrapower triple extension  $\pi^{\mathcal U}$ with the Aron--Berner extensions $\pi^i,$ ($i=0,\ldots,5)$. 

\begin{proposition}\label{ultra}
Let $(E,\pi)$ be a Jordan Banach triple system. Then the following statements hold:
 \begin{enumerate}[$(i)$]
\item For $m,n,p\in E^{**},$ if at least two of $m,n,p$ lie in $E$, then $\pi^{\mathcal U}(m,n,p)=\pi^i(m,n,p)$ for each $i=0,\ldots,5;$
\item Given $a\in E$, $n\in E^{**}$ and a bounded net $(c_\gamma)$ in $E$ $w^*$-converging to $p,$ then $$w^*\hbox{-}\lim_{\gamma} \pi^{\mathcal{U}} (a,n,c_{\gamma}) = \pi^1 (a,n,p);$$
\item Given $n\in E^{**}$ and bounded nets $(a_\alpha)$ and $(c_\gamma)$ in $E$ $w^*$-converging to $m$ and $p,$ respectively, then $$w^*\hbox{-}\lim_{\alpha} w^*\hbox{-}\lim_{\gamma} \pi^{\mathcal{U}} (a_{\alpha},n,c_{\gamma}) = w^*\hbox{-}\lim_{\alpha} \pi^1 (a_{\alpha},n,p) = \pi^1 (m,n,p);$$
\item If $\pi^{\mathcal U}$ is $w^*$-continuous in the first {\rm(}or equivalently, third{\rm)} variable, then $\pi^{\mathcal U}=\pi^1=\pi^4$.
\end{enumerate}
\end{proposition}

\begin{proof}
For $(i)$, we only prove $\pi^{\mathcal U}(a,b,p)=\pi^0(a,b,p),$ for $p\in E^{**}$ and $a,b\in E.$  The other cases follow by a similar argument. Let $J(p)=[c_\alpha]_{\mathcal U}\in E_{\mathcal U}$; then by the properties of the ultrafilter $\mathcal{U}$ we have $w^*\hbox{-}\lim_{\mathcal U}c_\alpha=p.$  Having in mind \eqref{u} and the fact that the map $\cdot\mapsto\pi^0(a,b,\cdot)$ is $w^*$-continuous, we deduce from Proposition \ref{B2} that
\begin{align*}
    \langle\pi^{\mathcal U}(a,b,p),\lambda\rangle =\lim_{\mathcal U}\langle\lambda,\pi(a,b,c_\alpha)\rangle=\lim_{\mathcal U}\langle\pi^0(a,b,c_\alpha),\lambda\rangle=\langle\pi^0(a,b,p),\lambda\rangle,
   \end{align*}
   for each $\lambda\in E^*,$ as claimed in $(i)$.\smallskip

$(ii)$ Let us fix $a\in E$, $n\in E^{**}$ and a bounded net $(p_\gamma)$ in $E$ $w^*$-converging to $p.$ It follows from $(i)$ and Proposition \ref{B2} that $$w^*\hbox{-}\lim_{\gamma} \pi^{\mathcal{U}} (a,n,p_{\gamma}) =w^*\hbox{-}\lim_{\gamma} \pi^{1} (a,n,p_{\gamma})= \pi^1 (a,n,p).$$

Statement $(iii)$ is a consequence of $(ii)$ and Proposition \ref{B2}.\smallskip

To prove $(iv)$, suppose that $\pi^{\mathcal U}$ is $w^*$-continuous in the first (and by symmetry, in the third) variable.  Let $m,n,p\in E^{**},$ and let $(m_\alpha), (p_\gamma)$ be bounded nets in $E$, $w^*$-converging to $m,p,$ respectively. Then by applying $(iii)$ and the hypothesis on $\pi^{\mathcal{U}}$ we get
   \[\pi^{\mathcal U}(m,n,p)=w^*\hbox{-}\lim_\alpha\lim_\gamma\pi^{\mathcal U}(m_\alpha,n,p_\gamma)=w^*\hbox{-}\lim_\alpha\lim_\gamma\pi^1(m_\alpha,n,p_\gamma)=\pi^1(m,n,p).\]
We thus deduce that $\pi^{\mathcal U}=\pi^1$. It follows that $\pi^1$ is symmetric in the outer variables and thus, by 
Proposition \ref{symmetry}, we have $\pi^{\mathcal U}=\pi^1=\pi^4.$
 \end{proof}


In the following example we compare triple products on the bidual of a Jordan Banach triple system obtained by Aron--Berner extensions with those obtained by ultrapower techniques.

\begin{example} \label{ex} The construction developed in this section can be also done for an ultrafilter $\mathcal{V}$ on an index set $I$ satisfying the property of the local reflexivity principle, that is there exists an isometric embedding $J$ of $E^{**}$ into $E_{\mathcal{V}}$ satisfying the following properties:\begin{enumerate}[$(a)$]\item  the restriction of $J$ to $E$ is the canonical embedding of $E$ into $E_{\mathcal{V}}$;
\item If $J(a) = [a_i]_{\mathcal{V}}$ with $a\in E^{**}$, then $w^*$-$\lim_{\mathcal{V}} (a_i) = a$.
\end{enumerate}

Following Example \ref{not}, let $E$ be a non Arens regular, unital   commutative  Banach $^*$-algebra (namely, the convolution group algebra $(\ell^1(\mathbb Z),\ast)$). Then the Jordan Banach triple system $(E,\pi),$ where $\pi(a,b,c)=ab^*c,$ is not  Aron--Berner regular.
Let $\mathcal V$ be any ultrafilter satisfying the property of the local reflexivity principle. By applying \eqref{u}, we deduce that the extension $\pi^{\mathcal V}$  is symmetric in the outer variables. But this is not the case for none of $\pi^i$ as it is explored in Example \ref{not}. Thus we  have  $\pi^{\mathcal V}\neq \pi^i$ for every $i=0,\ldots,5.$ Hence, by part $(iv)$ of Proposition \ref{ultra}, $\pi^{\mathcal V}$ is not $w^*$-continuous in the first variable. Therefore, there is no ultrafilter $\mathcal V$ satisfying the property of the local reflexivity principle such that the ultrapower extension $\pi^{\mathcal V}$ of $\pi$ is  $w^*$-continuous in the first variable; see also  Remark \ref{r new}.   We do not know if $(E^{**},\pi^{\mathcal V})$ is a Jordan Banach triple system for some such an ultrafilter $\mathcal V$.\smallskip

Let $E=\ell^1(\mathbb N)$. Then $E$ equipped with pointwise product is a commutative Banach $^*$-algebra which is  Arens regular and an ideal in its bidual (for details, see \cite[Example III.1]{GI} and/or \cite[Example 2.6.22 (iii)]{D}). Then $E$ can be viewed as an abelian Jordan Banach triple system with triple product $\pi(a,b,c)=\{a,b,c\}=ab^*c,$ (see \eqref {abelian}). Then the  Arens regularity of the Banach algebra $E$ implies that  $(E,\pi)$ is Aron--Berner regular; indeed,  for each $i=0,\ldots,5,$  we have $\pi^i(m,n,p)=m\Box n^*\Box p$ for every $m,n,p\in E^{**}.$ In particular, it is symmetric in all of its variables on $\big(E^{**}\big)_{sa}$. Suppose that $\mathcal V$ is an ultrafilter satisfying the property of the local reflexivity principle and such that $\pi^{\mathcal V}$ is not separately $w^*$-continuous (note that the existence of such an ultrafilter for $\ell^1(\mathbb N)$ is guaranteed by \cite[Example III.1.3]{GI}). We do not know if in this case $(E^{**},\pi^{\mathcal V})$ satisfies the Jordan identity.
\end{example}

An element $e$ in a Jordan Banach triple system $(E,\pi)$ is called \emph{tripotent} if $\pi(e,e,e)=e$. To simplify the notation, given $a,b\in E$ we write $L(a,b)$ for the linear mapping defined by $L(a,b) (x)= L^{\pi}(a,b) (x)= \pi (a,b,x)$ ($x\in E$). We shall denote by $Q(a,b)=Q^{\pi}(a,b)$ the conjugate linear map given by $Q^{\pi}(a,b) (x) = \pi (a,x,b)$. We shall write $Q(a)=Q^{\pi}(a)$ for $Q^{\pi}(a,a)$. 
\smallskip

It is known (see \cite[page 32]{Ch} or \cite[Theorem 3.13]{LoosIrvine}) that each tripotent $e$ in a Jordan Banach triple system $E$ induces a decomposition (called the Peirce decomposition) of $E$ as the direct sum of the eigenspaces of the mapping $L^{\pi}(e,e),$ that is,
$$E=E_0(e)\oplus E_1(e)\oplus E_2(e),$$
where $E_k(e)=\{x\in E : L^{\pi}(e,e)x=\frac{k}{2} x\}$ for $k=0,1,2.$ The space $E_k(e)$ is called the Peirce $k$-space associated with $e$. Peirce $k$-spaces satisfy the following (Peirce) multiplication rules:
\begin{enumerate}[$(1)$]
\item  $\{E_{i_{}}(e),E_{j_{}}(e),E_{k_{}}(e)\}\subseteq E_{i-j+k}(e),$
if $i-j+k\in\{0,1,2\}$ and $\{E_{i_{}}(e),E_{j_{}}(e),E_{k_{}}(e)\}=0$ otherwise;
\item  $\{E_0(e),E_2(e),E\}=\{E_2(e),E_0(e),E\}=0.$
\end{enumerate} As a consequence, Peirce $k$-spaces are subtriples. The projection $P_{k_{}}(e)$ of $E$ onto $E_{k_{}}(e)$ is called the Peirce $%
k$-projection of $e$. Peirce projections are given by the next formulae $$ P_{2_{}}(e) =Q^{\pi}(e)^2; \  P_{1_{}}(e) =2(L^{\pi}(e,e)-Q^{\pi}(e)^2); \ P_{0_{}}(e) =Id_E-2L^{\pi}(e,e)+Q^{\pi}(e)^2.$$ Actually $E_2(e)$ is a unital Jordan $^*$-algebra with unit $e$, Jordan product $a\circ_e b = \pi(a,e,b)$ and involution $a^{*_e} = Q^{\pi}(e) (a)$. The tripotent $e$ is called unitary if $E_2(e) = E$.\smallskip


Given $a,b$ in a Jordan triple system $(E,\pi)$ it follows from the Jordan identity that
\begin{equation}\label{eq Q Q} Q^{\pi}(a) Q^{\pi}(b) = 2 L^{\pi}(a,b) L^{\pi}(a,b) -L^{\pi}(Q^{\pi}(a)(b),b).
\end{equation} Let us pick $a,b\in E_2(e)$. We deduce from Peirce's rules that \begin{equation}\label{eq 2 at the extremes} Q^{\pi} (a,b) (x) = \pi(a,P_0(e) (x)\! +\! P_1(e) (x)\!+\! P_2 (e) (x),b) = \pi(a, P_2 (e) (x),b) = Q^{\pi} (a,b) P_2(e) (x),
\end{equation} for all $x\in E$.

\begin{proposition}\label{p unitary} Let $(E,\pi)$ be a Jordan Banach triple system. Suppose that the extension $\pi^0$ is symmetric in the outer variables. The the following statements hold:\begin{enumerate}[$(i)$]\item The mapping $Q^{\pi^0}(m) Q^{\pi^0}(n) : E^{**}\to E^{**}$ is weak$^*$-continuous for every $m,n\in E^{**}$;
\item If $e$ is a tripotent in $E$, then $Q^{\pi^0} (m,n): E^{**}\to E^{**}$ is weak$^*$-continuous for every $m,n\in (E^{**})_2(e)$;
\item If for each $m\in E^{**}$ there exists a tripotent $e\in E$ such that $m\in (E^{**})_2(e)$, then $\pi^0$ is separately weak$^*$-continuous; and hence $(E,\pi)$ is Aron--Berner regular;
\item If $u$ is a unitary tripotent in $E$, then $u$ is a unitary tripotent in $(E^{**},\pi^{0})$, $\pi^0$ is separately weak$^*$-continuous, and  $(E,\pi)$ is Aron--Berner regular.
\end{enumerate}
\end{proposition}

\begin{proof} Proposition \ref{symmetry} actually assures that $\pi^{0}=\pi^1=\pi^2=\pi^4$. We shall only deal with $\pi^0$.\smallskip

$(i)$ Since $\pi^0$ is symmetric in the outer variables, Theorem \ref{t outer symmetric is sufficient and necessary for pi0} assures that $(E^{**},\pi^0)$ is a Jordan Banach triple system. By construction the mappings $\pi^0 (\cdot,n,p)$, $\pi^0 (a,\cdot,p)$ and $\pi^0 (a,b,\cdot)$ are weak$^*$-continuous for all $n,p\in E^{**}$, $a,b\in E$. It follows from the hypothesis that the mapping $\pi^0 (p,n,\cdot)$ is weak$^*$-continuous for all $n,p\in E^{**}$. These properties guarantee that the mapping $L^{\pi^0}(n,p)$ is weak$^*$-continuous for all $n,p\in E^{**}$. Let us take $m,n\in E^{**}$, by \eqref{eq Q Q}, $Q^{\pi^0}(m)Q^{\pi^0}(n) = 2 L^{\pi^0}(m,n) L^{\pi^0}(m,n) -L^{\pi^0}(Q^{\pi^0}(m)(n),n)$ is weak$^*$-continuous too.\smallskip

$(ii)$ Suppose $e$ is a tripotent in $E$. Clearly $e$ is a tripotent in $(E^{**},\pi^{0})$. Since $e\in E$, the mapping $Q^{\pi^0}(e) = \pi^0(e,\cdot,e)$ is weak$^*$-continuous. Given $m\in (E^{**})_2(e)$, it follows from Peirce's rules (cf. \eqref{eq 2 at the extremes}) that $Q^{\pi^0} (m) = Q^{\pi^0} (m) P_2(e) = Q^{\pi^0} (m) Q^{\pi^0} (e) Q^{\pi^0} (e)$. Combining $(i)$ and the weak$^*$-continuity of $Q^{\pi^0}(e)$, we deduce that $Q^{\pi^0} (m)$ is weak$^*$-continuous. Now, given $m,n\in (E^{**})_2(e)$, by symmetry $Q^{\pi^0}(m,n) =\frac12( Q^{\pi^0}(m+n) - Q^{\pi^0}(m)-Q^{\pi^0}(n))$ is weak$^*$-continuous.\smallskip

$(iii)$ The assumptions imply that $\pi^0 (\cdot,m,n)=\pi^0 (n,m,\cdot)$ is weak$^*$-continuous for every $m,n\in E^{**}$. We deduce from the hypothesis and $(ii)$ that $Q(m)$ is weak$^*$-continuous for all $m\in E^{**}$, and thus the map $Q^{\pi^0}(m,n) =\frac12( Q^{\pi^0}(m+n) - Q^{\pi^0}(m)-Q^{\pi^0}(n))$ also is weak$^*$-continuous for every $m,n\in E^{**}$. Therefore $\pi^0 (m,\cdot,n)$ is weak$^*$-continuous for all $m,n\in E^{**}$, which shows that $\pi^0$ is separately weak$^*$-continuous. Remark \ref{r new} implies that $(E,\pi)$ is Aron--Berner regular.\smallskip

$(iv)$ Suppose $u$ is a unitary tripotent in $E$. Clearly $u$ is a tripotent in $(E^{**},\pi^0).$ We shall prove that $u$ is a unitary tripotent in $(E^{**},\pi^i)$. Having in mind that $u\in E$, the mapping $Q^{\pi^0}(u) = \pi^0(u,\cdot,u)$ is weak$^*$-continuous. We also know that $P_2(u)(x) = x$ for all $x\in E$ because $u$ is unitary. Combining these facts with Goldstine's theorem, we deduce that $P_2(u) (m) = Q^{\pi^0}(u) Q^{\pi^0}(u) (m) = m,$ for all $m\in E^{**}$, which proves that $u$ is unitary in $(E^{**},\pi^0).$ The desired conclusion is a consequence of $(iii)$ because $E^{**} = (E^{**})_2(u)$.
\end{proof}

We finish the paper by posing some questions which, to the best of
our knowledge,  seem to be open.

\begin{enumerate}[\hspace{1.5em}\rm(1)]
\item Is there a non Aron--Berner regular Jordan Banach triple system $(E,\pi)$ for which $(E^{**},\pi^{\mathcal U})$ is a Jordan Banach triple system? (See Example \ref{ex}(i)).
\item Let $(E,\pi)$  be a Jordan Banach triple system such that the extension $\pi^i$ is symmetric in the outer variables for  every  $i=0,\ldots,5.$ Is  $(E,\pi)$ Aron--Berner regular? (See Propositions \ref{l AB regularity implies the Jordan identity in the second dual} and  \ref{symmetry}).
\end{enumerate}\bigskip

\textbf{Acknowledgements} Third author partially supported by Junta de Andaluc\'{\i}a grant FQM375.


\bibliographystyle{amsplain}

\begin{thebibliography}{99}

\bibitem{Arr} {\sc R. Arens}, \textit{Operations induced in function classes,}  Monat. f\"{u}r Math. \textbf{55} (1951),  1--19.

\bibitem{Ar} {\sc R. Arens}, \textit{The adjoint of a bilinear operation}, Proc. Amer. Math. Soc. \textbf{2} (1951), 839--848.


\bibitem{AB} {\sc R.M.  Aron and P. Berner},  \textit{A Hahn-Banach extension theorem for analytic
mappings}, Bull. Soc. Math. France, \textbf{106} (1978), 3--24.

\bibitem{ArColGam91} {\sc R.M. Aron, B.J. Cole, T.W. Gamelin}, \textit{Spectra of algebras of analytic functions on a Banach space}, {J. Reine Angew. Math.} \textbf{415} (1991), 51--93.

\bibitem{AG} {\sc R.M. Aron and P.  Galindo}, \textit{ Weakly compact multilinear mappings},  Proc. Edinburgh Math. Soc. \textbf{40} (1997),  181--192.

\bibitem{BT}{\sc T. Barton and R.M. Timoney}, \textit{$weak^{\ast}$-continuity of Jordan triple products and its applications},  Math. Scand. \textbf{59} (1986), 177--191.



\bibitem{BomVi} {\sc F. Bombal, I. Villanueva}, \textit{Multilinear operators on
spaces of continuous functions}, \emph{Funct. Approx. Comment. Math.} \textbf{26} (1998), 117-126.

\bibitem{CGV} {\sc F. Cabello S\'{a}nchez, R. Garc\'{\i}a1 and I. Villanueva},  \textit{Extension of multilinear operators on Banach spaces}, Extracta Math., \textbf{15} (2000), 291--334.

\bibitem{Ch} {\sc C.-H. Chu}, \textit{Jordan structures in geometry and analysis}, Cambridge Text in Math. \textbf {190}, Cambridge Univ. Press, Cambridge, 2012.

\bibitem{CIL} {\sc C.-H. Chu, B. Iochum, and G. Loupias}, \textit{Grothendieck's theorem and factorization of
operators in Jordan triples},  Math. Ann. \textbf{284} (1989), 41--53.

\bibitem{ChuMe97} {\sc Ch.H. Chu, P. Mellon}, \textit{JB$^*$-triples have Pelczynski's Property V}, {Manuscripta Math.} \textbf{93} (1997), no. 3, 337-347.


\bibitem{D}{\sc H.G. Dales}, \textit{Banach algebras and automatic continuity}, London Math. Soc. Monographs \textbf{24}, Clarendon Press, Oxford, 2000.

\bibitem{DL}{\sc H.G. Dales, A.T. Lau}, \textit{The second duals of Beurling algebras}, Mem. Amer. Math. Soc. 177, American Mathematical Soc., Providence, Rhode Island 2005.

\bibitem{DG}{\sc A.M. Davie and T.W. Gamelin},  \textit{A theorem on polynomial-star approximation}, Proc. Amer. Math. Soc. \textbf{106} (1989), 351--356.

\bibitem{dineen1} {\sc S. Dineen}, \textit{The second dual of a JB$^*$-triple system}, Complex Analysis, Functional Analysis and Approximation Theory (Ed. J. Mujica), North-Holland, Amsterdam, 1986.

\bibitem{dineen2} {\sc S. Dineen}, \textit{Complete holomorphic vector fields on the second dual of a Banach space}, Math. Scand. \textbf{59} (1986), 131--142.




\bibitem{GI} {\sc G. Godefroy and B. Iochum}, \textit{ Arens-regularity of Banach algebras and the geometry of Banach spaces},
J. Funct. Anal. \textbf {80} (1988), 47--59.

\bibitem{GV} {\sc J.M. Guti\'{e}rrez and I. Villanueva},  \textit{Extensions of multilinear operators
and Banach space properties}, Proc. Royal Soc. Edinburgh \textbf{133A} (2003), 549--566.

\bibitem{He} {\sc S. Heinrich}, \textit{Ultraproducts in Banach space theory}, J. Reine Angew. Math. \textbf{313}  (1979),  72--104.


\bibitem{IL} {\sc B. Iochum \and G. Loupias}, \textit{Arens regularity and local reflexivity principle for Banach algebras,} Math. Ann. \textbf {284} (1989), 23--40 .

\bibitem{IL0} {\sc B. Iochum \and G. Loupias}, \textit{Remarks on the bidual of  Banach algebras (the C$^*$ case)}, Colloque Montpellier (1985). See also  Ann. Sci. Univ. Clermont-Ferrand II Math. \textbf{97} No. 27 (1991), 107--118.


\bibitem{K} {\sc W. Kaup}, \textit{A Riemann mapping theorem for bounded symmetric domains in complex Banach spaces},  Math. Z. \textbf{183} (1983), 503--529.

\bibitem{Ka84} {\sc W. Kaup}, \textit{Contractive projections on Jordan C$^*$-algebras and generalizations}, {Mathematica Scandinavica} \textbf{54} (1) (1984): 95-100.

\bibitem{LauUl96} A.T. Lau, A. \"Ulger, Topological centers of certain dual algebras, \emph{Trans. Amer. Math. Soc.} \textbf{348}(3) (1996), 1191-1212.

\bibitem{L} {\sc D.H. Leung}, \textit{Some remarks on regular Banach spaces}, Glasgow Math. J.  \textbf{38}
(1996), 243--248.

\bibitem{Lo} {\sc O. Loos}, \textit{Jordan Pairs}, Lecture Notes in Math. \textbf{460} Heidelberg: Springer-Verlag, 1975.

\bibitem{LoosIrvine}{\sc O. Loos}, \emph{Bounded symmetric domains and Jordan pairs}, Lecture Notes, Univ. California at Irvine,  1977.


\bibitem{martinez2000separate} {\sc J. Mart\'{i}nez and A.M. Peralta}, \textit{Separate weak$^*$-continuity of the triple product in dual real JB$^*$-triples}, Math. Z. \textbf{234} (2000), 635--646 .

\bibitem{MV} {\sc S. Mohammadzadeh and  H.R. Ebrahimi Vishki,} \textit{Arens regularity of module actions and the second adjoint of a derivation}, Bull. Austral. Math. Soc. \textbf{77} (2008), 465--476.

\bibitem{NME} {\sc M. Niazi, M.R. Miri and  H.R. Ebrahimi Vishki,} \textit{Ternary weak amenability of the bidual of a JB$^*$-triple}, Banach J. Math. Anal.  \textbf{11} (2017), 676--697

\bibitem{PP} {\sc A.M. Peralta and A. Rodr{\'i}guez-Palacios}, \textit{Grothendieck's inequalities for real and complex JBW$^*$-triples}, Proc. London Math. Soc. \textbf{83} (2001), 605--625.

\bibitem{PVMY} {\sc A.M. Peralta, I. Villanueva, J. D. Maitland Wright and K. Ylinen},  \textit{Quasi-completely continuous multilinear operators}, Proc. Royal Soc. Edinburgh \textbf{140A} (2010), 635--649.

\bibitem{Pfi94} {\sc H. Pfitzner}, \textit{Weak compactness in the dual of a C$^*$-algebra is determined commutatively}, {Math. Ann.} \textbf{298} (1994), no. 2, 349-371.

\bibitem{russo1994structure} {\sc B. Russo}, \textit{Structure of JB$^*$-triples}, in Jordan Algebras  (Oberwolfach,  1992), de Gruyter, Berlin, (1994),  209--280.

\bibitem{Sta82} {\sc L.L. Stach{\'o}}, \textit{A projection principle concerning biholomorphic automorphisms}, {Acta Sci. Math.} \textbf{44} (1982), 99-124.

\bibitem{T} {\sc M. Takesaki}, \textit{Theory of operator algebras I}, Encyclopaedia  Math. Sci. \textbf{124}, Springer-Verlag, 1979.


\bibitem{Z1} {\sc I. Zalduendo}, \textit{ Extending polynomials on Banach spaces - a survey},  Rev. Un. Mat. Argentina \textbf{46}  (2005),  45--72.

\end{thebibliography}

\end{document}